\documentclass[reqno]{amsart}

\usepackage{geometry}
\usepackage{enumerate}

\usepackage{amsmath}
\usepackage{amsfonts}
\usepackage{amssymb}

\usepackage{tikz}
\usetikzlibrary{positioning,arrows}

\newcommand{\R}{\mathbb R}
\newcommand{\C}{\mathbb C}

\newcommand{\Z}{\mathbb Z}
\newcommand{\Sp}{\mathbb S}

\newcommand{\eps}{\varepsilon}
\newcommand{\epss}{{\textstyle\frac{\varepsilon}{2}}}
\newcommand{\epsss}{\textstyle\frac{\varepsilon}{4}}
\newcommand{\halb}{\frac{1}{2}}
\newcommand{\Or}{{\mathcal O}}
\newcommand{\rep}{\operatorname{Re}}

\newcommand{\cal}{\mathcal}
\newcommand{\F}{{\mathfrak F}}
\newcommand{\No}{{\mathfrak N}}
\newcommand{\X}{{\mathfrak X}}
\newcommand{\Lo}{{\mathcal L}}

\newcommand{\sphere}{\mathbb{S}^2}
\renewcommand{\Re}{\operatorname{Re}}
\renewcommand{\Im}{\operatorname{Im}}

\usepackage{pict2e}

\DeclareRobustCommand{\ddiamond}{%
  \begingroup
  \setlength{\unitlength}{\fontcharht\font`T}%
  \begin{picture}(1,1)
  \polygon(.5,0)(1,.5)(.5,1)(0,.5)
  \polygon*(.5,0.2)(.8,.5)(.5,.8)(.2,.5)
  \end{picture}%
  \endgroup
}

\newtheorem{lemma}{Lemma}

\newtheorem{thm}{Theorem}

\newtheorem{dfn}{Definition}
\newtheorem{rmk}{Remark}

\title{Solution of the Bj\"orling problem by discrete approximation}

\author{Ulrike B\"ucking}
\author{Daniel Matthes}

\begin{document}

\begin{abstract}
  The Bj\"orling problem amounts to the construction of a minimal surface
  from a real analytic curve with a given normal vector field. 
  We approximate that solution locally by discrete minimal surfaces in  the 
spirit of~\cite{book}. 
  The main step in our construction is the approximation of the sought  
surface's Weierstra{\ss} data   by discrete conformal maps.
  We prove that the approximation error is of the order of the square of the 
mesh size.
\end{abstract}

\maketitle

\section{Introduction}
Minimal surfaces are classical objects in differential geometry which have been and still are studied in various aspects. For an (explicit) construction of minimal surfaces, several approaches are known. One possibility is to start with a real-analytic curve $\F_0:[a,b]\to\R^3$ with 
$\F_0'(t)\not=0$ 
for all $t\in[a,b]$ and a real-analytic vector field $\No_0:[a,b]\to\Sp^2$ with 
$\langle \No_0(t), \F_0'(t)\rangle =0$ for all $t\in[a,b]$. Assume that the maps $\No_0$ and $\F_0$ admit holomorphic extensions. The task of finding a minimal surface passing through the curve $\F_0$ and with 
given normals $\No_0$ along this curve is called {\em Bj\"orling problem for 
minimal surfaces}. It was proposed and solved by E.G.~Bj\"orling in~1844~\cite{B1844}.

Bj\"orling-type problems are now also known and solved for other classical surface classes like CMC-surfaces or for minimal surfaces in other space forms, like Lorentz-Minkowski space, see for example~\cite{ACM03,AV06,BD10,CDM11,Yang17,DFM17,BW18}. There is recent interest in using them for construction of special minimal surfaces, see~\cite{BO16, cho2023}. Also, Bj\"orling-type problems may be connected to other concepts as in~\cite{AM20}.

In this article, we are interested in solving Bj\"orling's problem locally via an explicit construction of discrete minimal surfaces as defined in~\cite{book}, see also \cite[Chapter~4.5]{BS08}. This definition relies on a discrete Weierstrass representation formula and thus on a discrete holomorphic function.
Therefore, the main task in our approach is to choose suitable data from the given real-analytic functions in order to determine initial values from which the discrete holomorphic function and eventually the discrete minimal surface are obtained.
We restrict ourselves to local considerations and to the generic case, that is, the given curve is nowhere tangent to a curvature line of the minimal surface. Our main focus is a suitable construction process from the given data which guarantees existence and convergence of the discrete minimal surfaces. In other words, we show how to extract data from the given real-analytic curves such that the corresponding discrete minimal surfaces locally approximate the unique smooth minimal surface solving the given Bj\"orling problem.

For our approach, we do not use the explicit formula for the smooth solution given by H.A.~Schwarz in~1890~\cite{Schwarz1890}. Instead, our construction is based on a special Weierstrass formula using a conformal curvature line parametrization, see for example~\cite{DHKW92}, whose main ingredient is the stereographic projection of the suitably parametrized Gauss map. In order to determine this holomorphic function in our setting, we first need a suitable reparametrization of the given functions $\F_0$ and $\No_0$ (see Section~\ref{secPhi}).  

We present two possibilities how to obtain corresponding discrete holomorphic functions $G_{m,n}$ (or more generally $G^\eps$). We make use of the notion of discrete holomorphicity based on the preservation the cross-ratios of an underlying rectangular lattice.  Our construction procedure is detailed in Sections~\ref{secConstruction} and~\ref{secConst2}. These discrete maps then define discrete minimal surfaces thanks to the discrete Weierstrass representation.

The main part of the paper is concerned with the proof that the discrete holomorphic functions $G^\eps$ obtained from our construction approximate the smooth holomorphic function $g$ with all its derivatives. Our proof relies on the convergence of suitable auxiliary discrete functions $F^\eps$ to a corresponding smooth function $f$. Their discrete and smooth evolution equations are derived in Section~\ref{secCR}. After proving convergence for these auxiliary functions in Section~\ref{secConvF}, we deduce the desired convergence of the discrete holomorphic functions $G^\eps$ in Sections~\ref{secConvMinimal}. This finally shows our approximation result for the discrete minimal surfaces, see Sections~\ref{secConstruction} and~\ref{secConst2}.

\section{From Bj\"orling data to Cauchy data for Weierstrass representation}\label{secPara}

In the following, we start from the data for the classical Bj\"orling problem and its solution in the well-known integral representation by H.A.~Schwarz. This solution may (locally) be rewritten in form of a Weierstrass representation: away from umbilic points, there exists a conformal curvature line parametrization which is determined by a holomorphic function. Our ultimate goal is to locally reformulate the given Bj\"orling problem as a corresponding Cauchy problem for a suitable discrete analogon of this holomorphic function.

\subsection{Representation formulas}\label{secRep}
For the classical {\em Bj\"orling problem for 
minimal surfaces} one assumes given a real-analytic curve $\F_0:[a,b]\to\R^3$ with derivative $\dot\F_0(t)\not=0$ 
for all $t\in[a,b]$ and a real-analytic normal vector field $\No_0:[a,b]\to\Sp^2$ with 
$\langle \No_0(t), \dot\F_0(t)\rangle =0$ for all $t\in[a,b]$. Moreover, the maps $\No_0$ and $\F_0$ admit holomorphic extensions. This data will be referred to as {\it Bj\"orling data}.

The task is to find a minimal surface passing through the curve $\F_0$ and with given normals $\No_0$.  Note that there always exists a local solution to the Bj\"orling problem.
In~1890, H.A.~Schwarz gave an explicit formula for the Weierstrass data of the solution, see for example~\cite{DHKW92}.
Denote by $\omega=t+i\eta$ a complex coordinate and by $\F_0'(\omega)$ the holomorphic extension of $\dot\F_0(t)$. Then
\[\X(t,\eta)=\rep\left( \F_0(\omega)-i\int_{a}^\omega \No_0(\omega)\times \F_0'(\omega)d\omega \right)\]
is a minimal surface $\X:D\to\R^3$ with normal vector field 
$\No:D\to\Sp^2$, both defined on some open domain $D\subset\C$ containing $[a,b]$, such that $\X(t,0)=\F_0(t)$ and $\No(t,0)=\No_0(t)$.

On the other hand, away from umbilic points every minimal surface can be locally parametrized by conformal curvature lines in the following form of a Weierstrass representation $\F:\Omega\to\R^3$ with
\begin{align}\label{eq:Weierstrass}
  \F_u=\Re\left[\frac1{g'}\rho(g)\right],
  \quad
  \F_v=-\Im\left[\frac1{g'}\rho(g)\right],
  \quad
  \No = \sigma(g),
\end{align}
where $g:\Omega\to\C$ is holomorphic on some open domain $\Omega\subset\C$,
and for $z\in\C$ we define
\begin{align}\label{eq:defsigma}
  \sigma(z) = \frac1{1+|z|^2}
  \begin{pmatrix}
    2\Re z \\ 2\Im z \\ |z|^2-1
  \end{pmatrix}
  \qquad \text{and}\qquad
  \rho(z) = 
  \begin{pmatrix}
    1-z^2 \\ i(1+z^2) \\ 2z
  \end{pmatrix} .
\end{align}

With the help of a suitable reparametrization, the minimal surface $\X$, which is the solution of the Björling problem, can of course locally be written in form of the Weierstrass representation~\eqref{eq:Weierstrass} (away from umbilic points).
The Bj\"orling data is attained in the sense that
there is a map $\phi:[a,b]\to\Omega$ with holomorphic extension 
such that
\begin{align}\label{eq:FF0}
  \F\circ\phi = \F_0, \quad \No\circ\phi=\No_0.
\end{align}

We are interested in the special representation formula~\eqref{eq:Weierstrass} for the minimal surface because
discrete minimal surfaces may be defined using an analogous Weierstrass representation formula, see for example~\cite{book,BHS06,Lam18}. 
Thus, our construction relies on discrete holomorphic functions $G_{m,n}$, which are in particular CR-mappings on rectangular lattices, see for example~\cite{book} or \cite[Chapter~4.5]{BS08}. 
Recall that the {\em cross-ratio} of four mutually distinct complex numbers $q_1, \dots , q_4 \in\C$ is defined as
\begin{equation*}
\text{CR}(q_1, q_2, q_3, q_4) = \frac{(q_1 - q_2)(q_3 - q_4)}{(q_2 - q_3)(q_4 - q_1)} . \label{eq:defcr}
\end{equation*}
Given a rectangular lattice in the complex plane whose vertices are labelled as $p_{m,n}$, a CR-mapping $G_{m,n}$ preserves the cross-ratios of all rectangles, see Def.~\ref{defCRmap} below. 
Given a CR-mapping $G_{m,n}$ in the complex plane, Bobenko and Pinkall showed how to obtain a {\em discrete minimal surface} $\F_{m,n}$ via a discrete Christoffel transformation of $G_{m,n}$, see~\cite{book,BS08}. 
In particular, the construction may be summarized by the following formulas:

\begin{align}
 \frac{\F_{m+1,n}-\F_{m,n}}{p_{m+1,n}-p_{m,n}} &=\Re\left[ 
\left(\frac{G_{m+1,n}-G_{m,n}}{p_{m+1,n}-p_{m,n}}\right)^{-1} \begin{pmatrix} 1-G_{m+1,n}G_{m,n}\\
i(1+G_{m+1,n}G_{m,n})\\ 
G_{m+1,n}+G_{m,n} \end{pmatrix} \right], \label{eq:discreteWeierstrass1} \\[1ex]
\frac{\F_{m,n+1}-\F_{m,n}}{-i(p_{m,n+1}-p_{m,n})} &= -\Im\left[ 
i\left(\frac{G_{m,n+1}-G_{m,n}}{p_{m,n+1}-p_{m,n}}\right)^{-1} \begin{pmatrix} 1-G_{m,n+1}G_{m,n}\\
i(1+G_{m,n+1}G_{m,n})\\
G_{m,n+1}+G_{m,n} \end{pmatrix}\right]. \label{eq:discreteWeierstrass2}\\[1ex]
\No_{m,n} &=\sigma(G_{m,n}) \notag
\end{align}
Note that the formulas are analogous to~\eqref{eq:Weierstrass}. The discrete minimal surface can easily be obtained from~\eqref{eq:discreteWeierstrass1}--\eqref{eq:discreteWeierstrass2} by discrete integration.
Lam showed in~\cite[Example~4]{Lam18} that these discrete surfaces are in fact minimal. 

Our goal is to describe an explicit construction based on the given Björling data how to obtain a discrete minimal surface which locally approximates the solution of the Björling problem. 
In particular, we aim at the following theorem.

\begin{thm}\label{theo:MinimalConv}
Given Björling data $\F_0$ and $\No_0$ and a point $\F_0(t_0)$ such that $\dot\F_0(t_0)$ is not parallel to $\dot\No_0(t_0)$, we can locally approximate the solution of the Björling problem $\F$  by discrete minimal surfaces $\F_{m,n}$.
These discrete minimal surfaces can be constructed from suitably chosen initial data obtained from the Bj\"orling data,  in particular from the reparametrization~$\phi$, see Section~\ref{secPhi} below, and the stereographic projection $G_0$ of the Gauss map $\No_0$. Details of our construction algorithm are given in Sections~\ref{secConstruction} and~\ref{secConst2}.

The convergence is in $C^\infty$, that is, all discrete derivatives also converge to their corresponding smooth counterparts.
\end{thm}

Our construction of the discrete minimal surfaces relies on the local Weierstrass representation~\eqref{eq:Weierstrass} and its discretization~\eqref{eq:discreteWeierstrass1}--\eqref{eq:discreteWeierstrass2}. The main ingredient for these representations are the holomorphic map $g$ and its discrete counterpart $G_{m,n}$.
Therefore, we first need an answer to the question how to determine $g$ and the reparametrization $\phi$ from the given data $\F_0$ and $\No_0$.

\subsection{Determination of $\phi$ and $g$ from $\F_0$ and $\No_0$}\label{secPhi}
Our first goal is to determine the function $\phi:[a,b]\to\Omega$ such that~\eqref{eq:FF0} holds. As the surface $\F$ is parametrized in conformal curvature line coordinates, we have in particular
\begin{align*}
 \F_u=\frac{\No_u}{\|\No_u\|^2},\qquad \F_v=-\frac{\No_v}{\|\No_v\|^2}, \qquad \|\F_u\|=\|\F_v\|,\qquad \langle \F_u,\F_v \rangle =0.
\end{align*}
Now we can deduce from~\eqref{eq:FF0} by straightforward calculations that
\begin{align*}
 \langle \dot\F_0, \dot \No_0 \rangle &= (\Re\dot \phi)^2 -(\Im\dot \phi)^2 ,\\
 \|\dot \F_0\|\|\dot\No_0\|&= (\Re\dot \phi)^2 +(\Im\dot \phi)^2=|\dot\phi|^2. 
\end{align*}
We immediately obtain
\begin{align}
& (\Re\dot \phi)^2 = \frac{1}{2}\left( \|\dot \F_0\|\|\dot\No_0\| -\langle \dot\F_0, \dot \No_0 \rangle \right)\qquad \text{and}\qquad
 (\Im\dot \phi)^2 = \frac{1}{2}\left( \|\dot \F_0\|\|\dot\No_0\| +\langle \dot\F_0, \dot \No_0 \rangle \right) \notag \\
&\text{and also}\qquad \dot\phi^2 =\|\dot \F_0\|\|\dot\No_0\|\text{e}^{\pm i\omega}=\langle \dot\F_0, \dot \No_0 \rangle \pm i\|\dot\F_0 \times \No_0 \| , \label{eq:phiwinkel}
\end{align}
where $\omega$ is the angle between $\dot\F_0$ and $\dot \No_0$. We assume that $\dot\F_0$ and $\dot \No_0$ are not parallel, that is, the given curve is not tangent to a curvature line of the minimal surface. Therefore, we have $\omega\not=0$ and the sign $\pm$ may be determined by the following considerations.

Let $G_0:[a,b]\to\C$ be the stereographic projection of $\No_0$, that is, $G_0$ is uniquely defined by
\begin{align*}
  \No_0=\sigma\circ G_0.
\end{align*}
Then by definition $g\circ\phi|_{[a,b]}=G_0$, and therefore
\begin{align*}
  g'\circ\phi\,\dot\phi = \dot G_0.
\end{align*}
Moreover, $\langle  \dot \F_0, \No_0\rangle=0$ and by~\eqref{eq:Weierstrass} and \eqref{eq:defsigma} we further deduce
\begin{align*}
  \dot \F_0 = \frac{d}{dt}(\F\circ\phi)
  = \F_u\circ\phi\,\Re\dot\phi + \F_v\circ\phi\,\Im\dot\phi
  = \Re\left[\frac{\dot\phi}{g'\circ\phi}\rho(g\circ\phi)\right]
  = \Re\left[\frac{\dot\phi^2}{\dot G_0}\rho( G_0)\right].
\end{align*}
This determines $\dot\phi^2$ and thus its square root $\dot\phi$ up to sign. By integration, we finally obtain $\phi$ uniquely up to translation. The map $\phi:[a,b]\to\Omega$ determines the curve of initial data in our new parameter space $\Omega$. As $G_0$ and $\F_0$ are real analytic, the same is true for $\phi$. Therefore, $\phi$ may be extended by analytic continuation to a neighborhood of $[a,b]$. 
Thus we have proven

\begin{lemma}\label{lemPhi}
 Assume given Björling data $\F_0$ and $\No_0$ such that 
$\dot\F_0$ and $\dot \No_0$ are not parallel for all $t\in[a,b]$. 
Then a map $\phi:[a,b]\to\Omega$ with holomorphic extension can be constructed from this data such that $\F\circ\phi = \F_0$ and $\No\circ\phi=\No_0$. A geometrically based relation of $\dot\phi^2$ to $\dot\No_0$ and $\dot\F_0$ is given in~\eqref{eq:phiwinkel}.

Furthermore, a corresponding holomorphic map $g$ can be determined such that $\No_0=\sigma\circ g\circ\phi$.
\end{lemma}

Note that for our construction of discrete minimal surfaces, as detailed in Section~\ref{secConstruction}, it suffices to know the values of $\phi$ on $[a,b]$. But our proof of convergence heavily relies on the fact, that $\phi$ possesses a holomorphic extension.

\subsection{Construction of rectangular lattices and discrete holomorphic functions}\label{secDisHolo}
According to our previous considerations we assume given 
an open domain $\Omega\subset\C$, a parametrized curve $\phi:[a,b]\to\Omega$ with holomorphic extension, and a bounded holomorphic function $g:\Omega\to\C$.
For our proof of convergence we assume that $g$ is known, 
but for the determination of suitable initial data for the discrete holomorphic function, it is actually sufficient to have access to the values of $g\circ\phi=G_0=\sigma\circ\No_0$, see Sections~\ref{secConstruction} and~\ref{secConst2}.

Without loss of generality, we assume that $0\in\Omega$ and $0\in [a,b]$. Furthermore, we restrict our local considerations to $t_0=0$ and to $\phi|_{[-\hat{a},\hat{a}]} =:\gamma:[-\hat{a},\hat{a}]\to\Omega$ with $t\mapsto u(t)+i  v(t)$, 
  normalized to $\gamma(0)=0$,
  where $u,v:D_{\hat{a}}\to\C$ are holomorphic functions on the complex disc $D_{\hat{a}}$ of radius $\hat{a}>0$,
  and are real for real arguments. 
  
  We assume that the trace of $\gamma$ considered as a graph in the $u$-$v$-plane is strictly monotone, that is
  the signs of $\dot u(t)$ and  $\dot v(t)$ do not change on $[-\hat{a},\hat{a}]$, for example $\dot u(t)\geq 0$ and  $\dot v(t)\geq 0$,
 and furthermore $\inf_{|t|<\hat{a}} \dot u(t)>0$ and 
$\inf_{|t|<\hat{a}} \dot v(t)>0$.
Therefore we have to exclude the cases where the given curve $\F_0$ contains a non-trivial part of a curvature line or is tangent to it. 
In terms of the curve $\gamma$, this means that we only consider the case when the the derivative $\frac{d}{dt}\gamma\in \R^2$ is contained in only one quadrant of $\R^2$ and is not parallel to any coordinate axis. 
For simplicity, we restrict ourselves in the following to the case where the vector $\frac{d}{dt}\gamma(0)$ lies in the first quadrant of $\R^2$. The remaining cases can be treated analogously. 

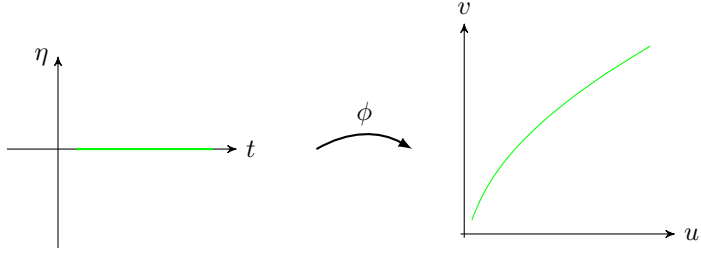
\begin{figure}
\begin{tikzpicture}[rotate=-45,scale=0.3]
          \draw[->, >=stealth'] (-1.7,-1.7)  -- (5.5, 5.5) node(tline)[right]
        {$t$};
        \draw[->, >=stealth'] (3.,-3.2)  -- (-3., 2.8) node(etaline)[left]
        {$\eta$};
\draw [thick, green] (0.095*5,0.095*5) -- (0.095*50,0.095*50);

\draw [thick,-latex,bend left] (8,8) to node (f) [above]{$\phi$} (11,11);
  \end{tikzpicture}\hspace{1em}
 \begin{tikzpicture}[scale=0.5]
    \def\xmin{0}
    \def\xmax{5}
    \def\ymin{0}
    \def\ymax{5}
    \def\xnum{10}
    \def\ynum{10}
    \draw[->, >=stealth'] (-0.2,-0.1)  -- (5.5,-0.1) node(xline)[right]
        {$u$};
         \draw[->, >=stealth'] (-0.1,-0.2)  -- (-0.1,5.5) node(yline)[above]
        {$v$};
    \draw [green] plot [smooth] coordinates {({\xmin+0.01*(2+0.077)*(\xmax - \xmin)},{\ymin + 0.049*(1+0.1)*(\ymax - \ymin)}) ({\xmin+0.01*(2*2+0.077*8)*(\xmax - \xmin)},{\ymin + 0.049*(2+0.1*4)*(\ymax - \ymin)}) ({\xmin+0.01*(2*3+0.077*27)*(\xmax - \xmin)},{\ymin + 0.049*(3+0.1*9)*(\ymax - \ymin)}) ({\xmin+0.01*(2*4+0.077*4*4*4)*(\xmax - \xmin)},{\ymin + 0.049*(4+0.1*16)*(\ymax - \ymin)}) ({\xmin+0.01*(2*5+0.077*5*5*5)*(\xmax - \xmin)},{\ymin + 0.049*(5+0.1*25)*(\ymax - \ymin)}) ({\xmin+0.01*(2*6+0.077*6*6*6)*(\xmax - \xmin)},{\ymin + 0.049*(6+0.1*36)*(\ymax - \ymin)}) ({\xmin+0.01*(2*7+0.077*7*7*7)*(\xmax - \xmin)},{\ymin + 0.049*(7+0.1*49)*(\ymax - \ymin)}) ({\xmin+0.01*(2*8+0.077*8*8*8)*(\xmax - \xmin)},{\ymin + 0.049*(8+0.1*64)*(\ymax - \ymin)}) ({\xmin+0.01*(2*9+0.077*9*9*9)*(\xmax - \xmin)},{\ymin + 0.049*(9+0.1*81)*(\ymax - \ymin)}) ({\xmin+0.01*(2*10+0.077*10*10*10)*(\xmax - \xmin)},{\ymin + 0.049*(10+0.1*100)*(\ymax - \ymin)})};
  \end{tikzpicture}
  \caption{Example of the parametrized curve $\phi|_{[-\hat{a},\hat{a}]} =:\gamma:[-\hat{a},\hat{a}]\to\Omega$ with $t\mapsto u(t)+i  v(t)$ which gives rise to a coordinate transformation from the $(t,\eta)$-plane to the $(u,v)$-plane.}\label{fig:phi}
  \end{figure}

Our approach extensively uses the {\it coordinate transformation}
\begin{equation}\label{eq:p1}
p(t,\eta)=u(t-\eta)+iv(t+\eta)
\end{equation}
based on the functions $u$ and $v$ as indicated in Figures~\ref{fig:phi} and~\ref{fig:mesh}.
For a given $\eps>0$, we now define a discrete parameter space based on a equidistant sampling of $p$ as
 a {\it rectangular lattice} $\Omega^\eps\subset\Omega$ with points (for suitable $m,n\in\Z$)
 \begin{align}\label{eq:pmn}
  p_{m,m}=u(m\eps)+iv(m\eps)=\phi(m\eps)\in\gamma \qquad \text{and}\qquad
  p_{m,n}= 
u(m\eps)+iv(n\eps).
 \end{align}

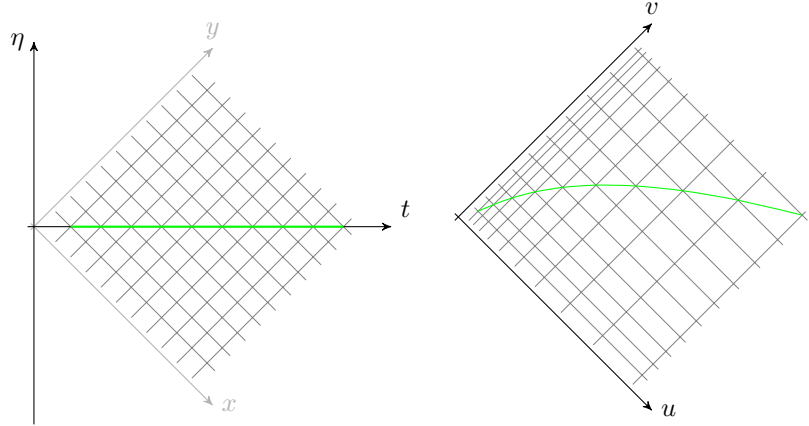
\begin{figure}
\begin{tikzpicture}[rotate=-45,scale=0.6]
    \def\xmin{0}
    \def\xmax{5}
    \def\ymin{0}
    \def\ymax{5}
    \def\xnum{10}
    \def\ynum{10}
    \draw[help lines]
      \foreach \i in {1, ..., \ynum} {
        (\xmin, {\ymin + 0.095*(\i)*(\ymax - \ymin)})
        -- ++(\xmax - \xmin, 0)
      }
      \foreach \i in {1, ..., \xnum} {
({\xmin + 0.095*(\i)*(\xmax - \xmin)}, \ymin)
        -- ++(0, \ymax - \ymin)
      }
    ;
    \draw[->, >=stealth',black!30] (-0.2,-0.1)  -- (5.5,-0.1) node(xline)[right]
        {$x$};
         \draw[->, >=stealth',black!30] (-0.1,-0.2)  -- (-0.1,5.5) node(yline)[above]
        {$y$};
          \draw[->, >=stealth'] (-0.2,-0.2)  -- (5.5, 5.5) node(tline)[above right]
        {$t$};
        \draw[->, >=stealth'] (3.,-3.2)  -- (-3., 2.8) node(etaline)[left]
        {$\eta$};
\draw [thick, green] (0.095*5,0.095*5) -- (0.095*50,0.095*50);
  \end{tikzpicture}\hspace{1em}
 \begin{tikzpicture}[rotate=-45,scale=0.65]
    \def\xmin{0}
    \def\xmax{5}
    \def\ymin{0}
    \def\ymax{5}
    \def\xnum{10}
    \def\ynum{10}
    \pgfmathsetseed{345}
    \draw[help lines]
      \foreach \i in {1, ..., \ynum} {
        (\xmin, {\ymin + 0.049*(\i+0.1*\i*\i)*(\ymax - \ymin)})
        -- ++(\xmax - \xmin, 0)
      }
      \foreach \i in {1, ..., \xnum} {
({\xmin + 0.01*(2*\i+0.077*\i*\i*\i)*(\xmax - \xmin)}, \ymin)
        -- ++(0, \ymax - \ymin)
      }
    ;
    \draw[->, >=stealth'] (-0.2,-0.1)  -- (5.5,-0.1) node(xline)[right]
        {$u$};
         \draw[->, >=stealth'] (-0.1,-0.2)  -- (-0.1,5.5) node(yline)[above]
        {$v$};
    \draw [green] plot [smooth] coordinates {({\xmin+0.01*(2+0.077)*(\xmax - \xmin)},{\ymin + 0.049*(1+0.1)*(\ymax - \ymin)}) ({\xmin+0.01*(2*2+0.077*8)*(\xmax - \xmin)},{\ymin + 0.049*(2+0.1*4)*(\ymax - \ymin)}) ({\xmin+0.01*(2*3+0.077*27)*(\xmax - \xmin)},{\ymin + 0.049*(3+0.1*9)*(\ymax - \ymin)}) ({\xmin+0.01*(2*4+0.077*4*4*4)*(\xmax - \xmin)},{\ymin + 0.049*(4+0.1*16)*(\ymax - \ymin)}) ({\xmin+0.01*(2*5+0.077*5*5*5)*(\xmax - \xmin)},{\ymin + 0.049*(5+0.1*25)*(\ymax - \ymin)}) ({\xmin+0.01*(2*6+0.077*6*6*6)*(\xmax - \xmin)},{\ymin + 0.049*(6+0.1*36)*(\ymax - \ymin)}) ({\xmin+0.01*(2*7+0.077*7*7*7)*(\xmax - \xmin)},{\ymin + 0.049*(7+0.1*49)*(\ymax - \ymin)}) ({\xmin+0.01*(2*8+0.077*8*8*8)*(\xmax - \xmin)},{\ymin + 0.049*(8+0.1*64)*(\ymax - \ymin)}) ({\xmin+0.01*(2*9+0.077*9*9*9)*(\xmax - \xmin)},{\ymin + 0.049*(9+0.1*81)*(\ymax - \ymin)}) ({\xmin+0.01*(2*10+0.077*10*10*10)*(\xmax - \xmin)},{\ymin + 0.049*(10+0.1*100)*(\ymax - \ymin)})};
  \end{tikzpicture}
  \caption{Example of the lattice $\Omega^\eps$ (right) and the curve $\gamma=\phi|_{[-\hat{a},\hat{a}]}$ (right, colored green) passing through the lattice points $p_{n,n}$.}\label{fig:mesh}
  \end{figure}

\begin{rmk}\label{remuv}
The rectangular lattice $\Omega^\eps$ is defined using points on the curve $\gamma=\phi|_{[-\hat{a},\hat{a}]}$. 
Alternatively, we could start with only one point on the curve and the recursive definition of the other mesh points by 
\begin{align*}
p_{m,n}-p_{m-1,n}&=\eps \dot u(m\eps-\epss)\qquad \text{and}\qquad p_{m,n+1}-p_{m,n}= \eps i\dot v(n\eps+\epss). 
\end{align*}
This amounts to a construction of the rectangular lattice from discretizations of $\dot u$ and $\dot v$ by discrete local integration.
Also in the following considerations and proofs, it is sufficient to know $\dot u$ and $\dot v$ on $[-\hat{a},\hat{a}]$. In particular, we could replace $p_{m,n}-p_{m-1,n}$ by $\eps
\dot u(m\eps-\epss)$ and $p_{m,n+1}-p_{m,n}$ by $\eps i\dot v(n\eps+\epss)$ in all 
formulas below.  
\end{rmk}

Given a rectangular lattice $\Omega^\eps$,
we define a discrete holomorphic function 
as the discrete map which preserves the cross-ratios of all rectangles, see Figure~\ref{fig:CRmap} for an example. 

\begin{dfn}[{\cite[Def.~21]{book}, see also~\cite[Chap.~8]{BS08}}]\label{defCRmap}
A map $G:\Omega^\eps\to\C$ is called {\em discrete conformal} or {\em CR-mapping} if 
\begin{equation}\label{eq:defdiscconf}
\text{CR}(G_{m-1,n},G_{m,n},G_{m,n+1},G_{m-1,n+1})= \text{CR}(p_{m-1,n},p_{m,n},p_{m,n+1},p_{m-1,n+1})
\end{equation}

holds for all rectangles of $\Omega^\eps$, where the cross-ratio is defined in ~\eqref{eq:defcr}.
\end{dfn}

\begin{figure}[htb]
  \includegraphics[height=2cm]{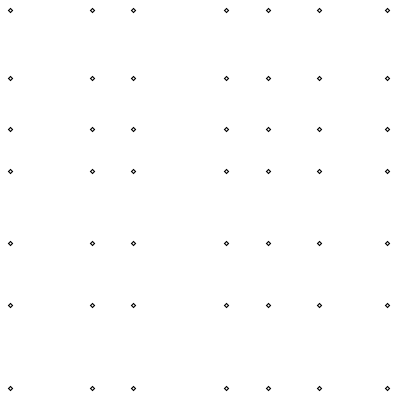}\hspace{2em} 
  \raisebox{3ex}{$\stackrel{G_{m,n}}{\longrightarrow}$} \hspace{2em}
  \includegraphics[height=3cm]{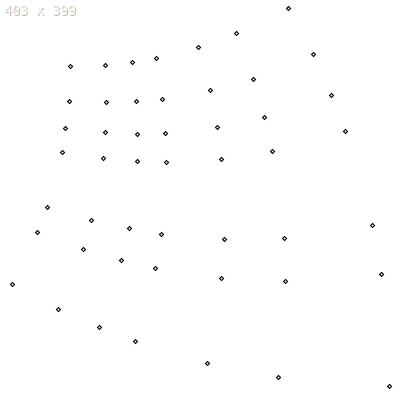}
  \caption{Example of a CR-mapping $G_{m,n}$ on a rectangular lattice}\label{fig:CRmap}
\end{figure}

\begin{figure}[htb]
\begin{tikzpicture}[rotate=-45,scale=0.75]
    \def\xmin{0}
    \def\xmax{5}
    \def\ymin{0}
    \def\ymax{5}
    \def\xnum{10}
    \def\ynum{10}
    \pgfmathsetseed{345}
    \draw[help lines]
      \foreach \i in {1, ..., \ynum} {
        (\xmin, {\ymin + 0.049*(\i+0.1*\i*\i)*(\ymax - \ymin)})
        -- ++(\xmax - \xmin, 0)
      }
      \foreach \i in {1, ..., \xnum} {
({\xmin + 0.01*(2*\i+0.077*\i*\i*\i)*(\xmax - \xmin)}, \ymin)
        -- ++(0, \ymax - \ymin)
      }
    ;
    \filldraw[black] ({\xmin+0.01*(2+0.077)*(\xmax - \xmin)},{\ymin + 0.049*(1+0.1)*(\ymax - \ymin)})  circle (1.5pt);
     \filldraw[black] ({\xmin+0.01*(2+0.077)*(\xmax - \xmin)},{\ymin + 0.049*(2+0.1*4)*(\ymax - \ymin)})  circle (1.5pt);
      \filldraw[black] ({\xmin+0.01*(2*2+0.077*8)*(\xmax - \xmin)},{\ymin + 0.049*(2+0.1*4)*(\ymax - \ymin)}) circle (1.5pt);
       \filldraw[black] ({\xmin+0.01*(2*2+0.077*8)*(\xmax - \xmin)},{\ymin + 0.049*(3+0.1*9)*(\ymax - \ymin)}) circle (1.5pt);
    \filldraw[black]  ({\xmin+0.01*(2*3+0.077*27)*(\xmax - \xmin)},{\ymin + 0.049*(3+0.1*9)*(\ymax - \ymin)}) circle (1.5pt);
    \filldraw[black]  ({\xmin+0.01*(2*3+0.077*27)*(\xmax - \xmin)},{\ymin + 0.049*(4+0.1*16)*(\ymax - \ymin)}) circle (1.5pt);
    \filldraw[black] ({\xmin+0.01*(2*4+0.077*4*4*4)*(\xmax - \xmin)},{\ymin + 0.049*(4+0.1*16)*(\ymax - \ymin)}) circle (1.5pt);
    \filldraw[black] ({\xmin+0.01*(2*4+0.077*4*4*4)*(\xmax - \xmin)},{\ymin + 0.049*(5+0.1*25)*(\ymax - \ymin)}) circle (1.5pt);
    \filldraw[black] ({\xmin+0.01*(2*5+0.077*5*5*5)*(\xmax - \xmin)},{\ymin + 0.049*(5+0.1*25)*(\ymax - \ymin)}) circle (1.5pt);
     \filldraw[black] ({\xmin+0.01*(2*5+0.077*5*5*5)*(\xmax - \xmin)},{\ymin + 0.049*(6+0.1*36)*(\ymax - \ymin)}) circle (1.5pt);
    \filldraw[black] ({\xmin+0.01*(2*6+0.077*6*6*6)*(\xmax - \xmin)},{\ymin + 0.049*(6+0.1*36)*(\ymax - \ymin)}) circle (1.5pt);
    \filldraw[black] ({\xmin+0.01*(2*6+0.077*6*6*6)*(\xmax - \xmin)},{\ymin + 0.049*(7+0.1*49)*(\ymax - \ymin)}) circle (1.5pt);
    \filldraw[black]  ({\xmin+0.01*(2*7+0.077*7*7*7)*(\xmax - \xmin)},{\ymin + 0.049*(7+0.1*49)*(\ymax - \ymin)}) circle (1.5pt);
    \filldraw[black]  ({\xmin+0.01*(2*7+0.077*7*7*7)*(\xmax - \xmin)},{\ymin + 0.049*(8+0.1*64)*(\ymax - \ymin)}) circle (1.5pt);
    \filldraw[black] ({\xmin+0.01*(2*8+0.077*8*8*8)*(\xmax - \xmin)},{\ymin + 0.049*(8+0.1*64)*(\ymax - \ymin)}) circle (1.5pt);
    \filldraw[black] ({\xmin+0.01*(2*8+0.077*8*8*8)*(\xmax - \xmin)},{\ymin + 0.049*(9+0.1*81)*(\ymax - \ymin)}) circle (1.5pt);
     \filldraw[black] ({\xmin+0.01*(2*9+0.077*9*9*9)*(\xmax - \xmin)},{\ymin + 0.049*(9+0.1*81)*(\ymax - \ymin)}) circle (1.5pt);
      \filldraw[black] ({\xmin+0.01*(2*9+0.077*9*9*9)*(\xmax - \xmin)},{\ymin + 0.049*(10+0.1*100)*(\ymax - \ymin)}) circle (1.5pt);
       \filldraw[black]  ({\xmin+0.01*(2*10+0.077*10*10*10)*(\xmax - \xmin)},{\ymin + 0.049*(10+0.1*100)*(\ymax - \ymin)}) circle (1.5pt);
        
        \draw [green] plot [smooth] coordinates {({\xmin+0.01*(2+0.077)*(\xmax - \xmin)},{\ymin + 0.049*(1+0.1)*(\ymax - \ymin)}) ({\xmin+0.01*(2*2+0.077*8)*(\xmax - \xmin)},{\ymin + 0.049*(2+0.1*4)*(\ymax - \ymin)}) ({\xmin+0.01*(2*3+0.077*27)*(\xmax - \xmin)},{\ymin + 0.049*(3+0.1*9)*(\ymax - \ymin)}) ({\xmin+0.01*(2*4+0.077*4*4*4)*(\xmax - \xmin)},{\ymin + 0.049*(4+0.1*16)*(\ymax - \ymin)}) ({\xmin+0.01*(2*5+0.077*5*5*5)*(\xmax - \xmin)},{\ymin + 0.049*(5+0.1*25)*(\ymax - \ymin)}) ({\xmin+0.01*(2*6+0.077*6*6*6)*(\xmax - \xmin)},{\ymin + 0.049*(6+0.1*36)*(\ymax - \ymin)}) ({\xmin+0.01*(2*7+0.077*7*7*7)*(\xmax - \xmin)},{\ymin + 0.049*(7+0.1*49)*(\ymax - \ymin)}) ({\xmin+0.01*(2*8+0.077*8*8*8)*(\xmax - \xmin)},{\ymin + 0.049*(8+0.1*64)*(\ymax - \ymin)}) ({\xmin+0.01*(2*9+0.077*9*9*9)*(\xmax - \xmin)},{\ymin + 0.049*(9+0.1*81)*(\ymax - \ymin)}) ({\xmin+0.01*(2*10+0.077*10*10*10)*(\xmax - \xmin)},{\ymin + 0.049*(10+0.1*100)*(\ymax - \ymin)})};
  \end{tikzpicture}
 \caption{Example of an initial 'zig-zag'-curve in parameter space (black points), containing points of the given curve (green), for the evolution of CR-mappings}\label{fig:zigzag}
\end{figure}
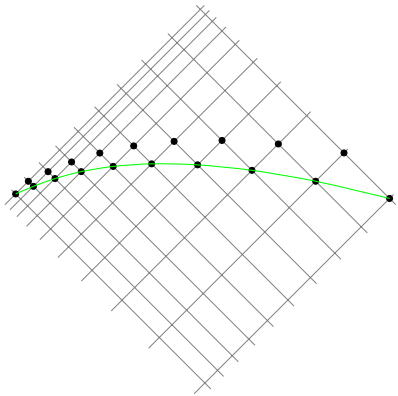

\begin{rmk}
 Replacing the differences of the lattice by derivatives in the spirit of Remark~\ref{remuv}, we have to replace the defining condition~\eqref{eq:defdiscconf} by
 \begin{align*}
  \text{CR}(G_{m-1,n},G_{m,n},G_{m,n+1},G_{m-1,n+1})= -\left(\frac{\dot u(m\eps-\epss)}{\dot v(n\eps+\epss)}\right)^2.
 \end{align*}
\end{rmk}

There is a natural way to construct CR-mappings from initial Cauchy data by prescribing the values for a {\it 'zig-zag'-curve} in parameter space as indicated in Figure~\ref{fig:zigzag}. For generic values, these data can be extended to a CR-mapping in a unique way, that is, all other values are easily obtained inductively from the prescribed data on the 'zig-zag'-curve.

 A solution for~\eqref{eq:defdiscconf} may also be obtained from other discrete curves in the lattice than a proper 'zig-zag'. But our numerical evidence suggests that the 'zig-zag' leads to better results and for other curves the solutions may diverge quickly. This is the reason why we have introduced the non-equidistant rectangular lattice $\Omega^\eps$.
 Furthermore, the necessity of an initial 'zig-zag'-curve requires the given curve $\gamma=\phi|_{[-a,a]}$ not to contain a part of a curvature line (corresponding to the parameter lines $u=const$ and $v=const$).
 
 Consequently, the main difficulty is to come up with appropriate initial data. As already noted in~\cite{BP96}, discrete conformal maps depend very sensitively on their initial data. For inappropriate choices of data on the initial 'zig-zag'-curve the sequence of discrete conformal maps may diverge rapidly.

Given a CR-mapping $G_{m,n}$, a {discrete minimal surface} $\F_{m,n}$ can be obtained by~\eqref{eq:discreteWeierstrass1}--\eqref{eq:discreteWeierstrass2}. 
Therefore, our main aim is to construct a CR-mapping $G_{m,n}$ from suitably chosen initial data on the given curve $\gamma$ and then establish its convergence to the given holomorphic map $g$. 
This leads to the convergence of the corresponding minimal surfaces by~\eqref{eq:Weierstrass} and~\eqref{eq:discreteWeierstrass1}--\eqref{eq:discreteWeierstrass2} respectively.

\begin{thm}\label{theo:ConvG}
 There exist suitable initial values for $G_{m,m}$ and $G_{m,m+1}$ which can be obtained from the Björling data (including derivatives of the data), such that the corresponding CR-mappings, which solve the Cauchy problem for~\eqref{eq:defdiscconf} and the given initial values, exist locally in a neighborhood of $0$  and converge to $g$ in $C^\infty$, that is all discrete derivatives also converge to their corresponding smooth counterparts. 
\end{thm}

Suitable initial values are detailed in the following section and in Section~\ref{secConst2}. The proof of Theorem~\ref{theo:ConvG} relies on auxialary functions introduced in Section~\ref{secCR} and is presented in Section~\ref{secConvMinimal}.

\subsection{Construction of Cauchy data for $G_{m,n}$ for the proof of Theorem~\ref{theo:MinimalConv}}\label{secConstruction}
In the following we explain explicitely how to locally construct discrete minimal surfaces which approximate the solution of the Bj\"orling problem.
Starting from  Bj\"orling data in the form
\begin{align*}
  \F_0:[a,b]\to\R^3, \quad \No_0:[a,b]\to\sphere \text{ with } \langle \F_0',\No_0\rangle = 0,
\end{align*}
where the maps $\No_0$ and $\F_0$ admit holomorphic extensions, we first restrict our considerations to a neighborhood $U_0$ of $t_0=0$ such that the restricted curve $\F_0|_{U_0}$ is nowhere tangent to a curvature line (i.e.\ $\dot\No_0$ should not be parallel to $\dot\F_0$). Possibly, we further restrict this neighborhood to obtain the desired convergence, see Section~\ref{secConvF}.

Suitable Cauchy data for $G_{m,n}$, which guarantees the convergence of the corresponding discrete minimal surfaces, may be obtained in different ways. We present here one possibility which uses values of $G_0=\sigma\circ \No_0$ directly as initial values for $G_{m,m}$ and additional values for $G_{m,m+1}$ derived from the Björling data.  Another possibility is detailled in Section~\ref{secConst2}. 

\begin{enumerate}[(i)]
 \item 
Thanks to the relations detailed in Section~\ref{secPhi}, we can pass from the original functions $\F_0$ and $\No_0$ to the map $\phi$, noting the geometric relation~\eqref{eq:phiwinkel} for $\dot\phi^2$.
\item We now choose a parameter $\eps$ and determine 
\begin{align*}
 \dot\phi^2_{m,m}& :=\dot\phi^2(m\eps)
\end{align*}
and from its square root $\dot \phi=\dot u+i\dot v$ also the values (in between mesh points $m\eps$)
\begin{align*}
 \dot u_{m+\halb} &= \dot u(m\eps+\epss) &\text{and}&& \dot v_{m+\halb} &= \dot v(m\eps+\epss).
\end{align*}
Fixing one value $p_{0,0}$ we obtain the rectangular lattice $p_{m,n}$ by discrete integration as in Remark~\ref{remuv} using 
\begin{align*}
p_{m,n}-p_{m-1,n}&=\eps \dot u_{m+\halb} &\text{and}&& p_{m,n+1}-p_{m,n}&= \eps i\dot v_{m+\halb}.
\end{align*}
 \item 
We directly read off all initial values for $m=n$ from the given function $G_0=\sigma\circ\No_0$.
\begin{align}
 G_{m,m}& := G_0(\eps m) = \sigma\circ\No_0(\eps m) \label{eq:defGm1}
 \end{align}
 \item
 Using the derivative $\dot G_0$ of $G_0=\sigma\circ\No_0$ we also read off
 \begin{align*}
 \dot G_{m,m}& := \dot G_0(\eps m) .
 \end{align*}
  The missing initial values $G_{m,m+1}$ are obtained as an extrapolation based on~\eqref{eq:defGm1} and the values of $\dot\phi^2$ and $\dot\phi$ in (ii) by
  \begin{align}
   G_{m,m+1} &:= \frac{\dot u_{m+\halb}G_{m,m} \sqrt[4]{\dot\phi^2_{m,m}\dot G_{m+1,m+1}^2} +i\dot v_{m+\halb}G_{m+1,m+1}\sqrt[4]{\dot\phi^2_{m+1,m+1}\dot G_{m,m}^2}}{\dot u_{m+\halb} \sqrt[4]{\dot\phi^2_{m,m}\dot G_{m+1,m+1}^2} +i\dot v_{m+\halb}\sqrt[4]{\dot\phi^2_{m+1,m+1}\dot G_{m,m}^2}} .\label{eq:defGm2}
  \end{align}

\item 
Evolution by~\eqref{eq:defdiscconf} starting from our initial 'zig-zag'-curve (see Figure~\ref{fig:zigzag}) now produces a CR-mapping $G_{m,n}$ in a neighborhood of $G_0(0)$. The local existence is guaranteed by Theorem~\ref{theo:ConvG}. 
\item 
From this CR-mapping $G_{m,n}$ we construct a discrete minimal surface $\F_{m,n}$ for a suitably chosen starting point $\F_{0,0}=\F_0(0)$ from discrete integration of~\eqref{eq:discreteWeierstrass1}--\eqref{eq:discreteWeierstrass2}.
\end{enumerate}
 As $G_{m,n}$ approximates the smooth function $g$ locally in $C^\infty$ with error of order $\eps^2$ by Theorem~\ref{theo:ConvG},
 we deduce from the smooth Weierstrass representation~\eqref{eq:Weierstrass} that
\begin{align*}
 \frac{\F_{m+1,n}-\F_{m,n}}{p_{m+1,n}-p_{m,n}}&= 
\Re\left[\frac{1}{g'}\rho(g)\right] +\Or(\eps^2)= \F_x+\Or(\eps^2), \\
\frac{\F_{m,n+1}-\F_{m,n}}{-i(p_{m,n+1}-p_{m,n})}&= 
-\Im\left[\frac{1}{g'}\rho(g)\right] +\Or(\eps^2)= \F_y+\Or(\eps^2) .
\end{align*}
This shows that the discrete minimal surface $\F_{m,n}$ locally approximates the smooth minimal surface $\F$ with an error of order~$\eps^2$ and thus proves Theorem~\ref{theo:MinimalConv}.

\begin{figure}[tb]
 \includegraphics[trim=1cm 0 1cm 0, clip,width=0.65\textwidth]{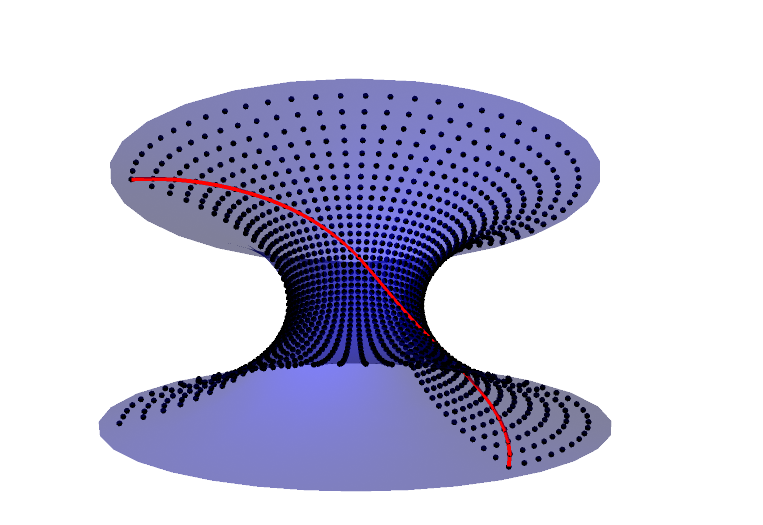}
  \includegraphics[width=0.3\textwidth]{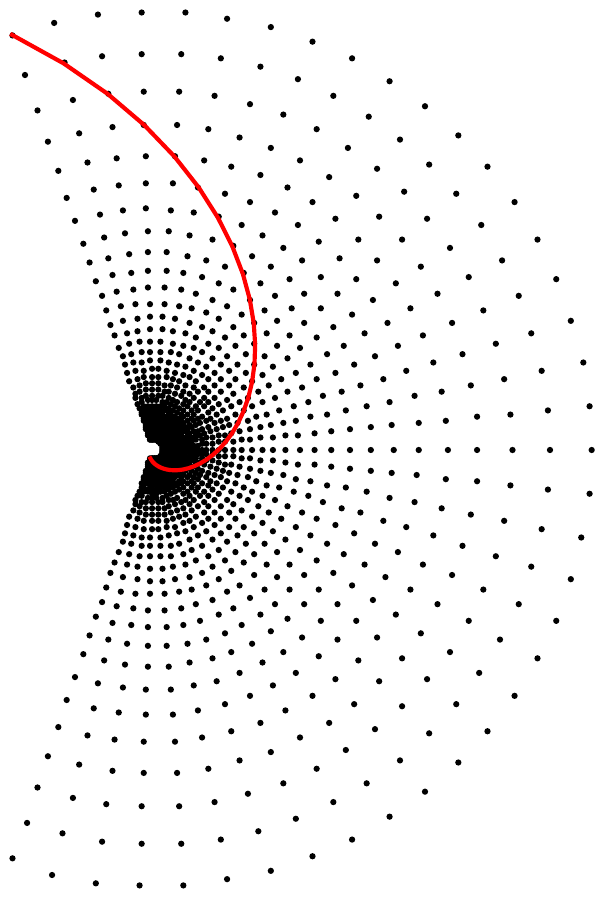}
 \caption{Left: Given curve $\F_0$ (red) and black points of the discrete minimal surface $\F_{m,n}$ obtained from our construction procedure in Section~\ref{secConstruction}. Right: Image of the discrete holomorphic map $G_{m,n}$ (black points) and curve of initial values $G_0$ (red). }\label{fig:Ex1}
\end{figure}

\begin{figure}[tb]
 \includegraphics[trim=1cm 2cm 1cm 0, clip,width=0.7\textwidth]{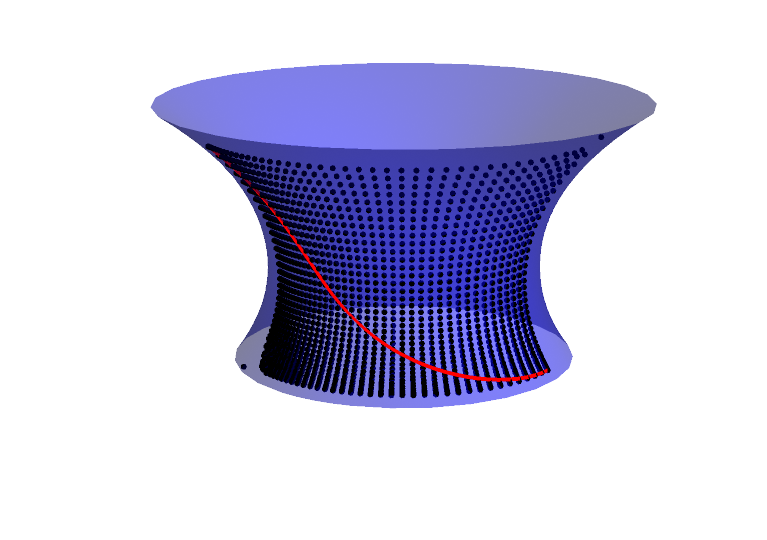}
  \includegraphics[width=0.25\textwidth]{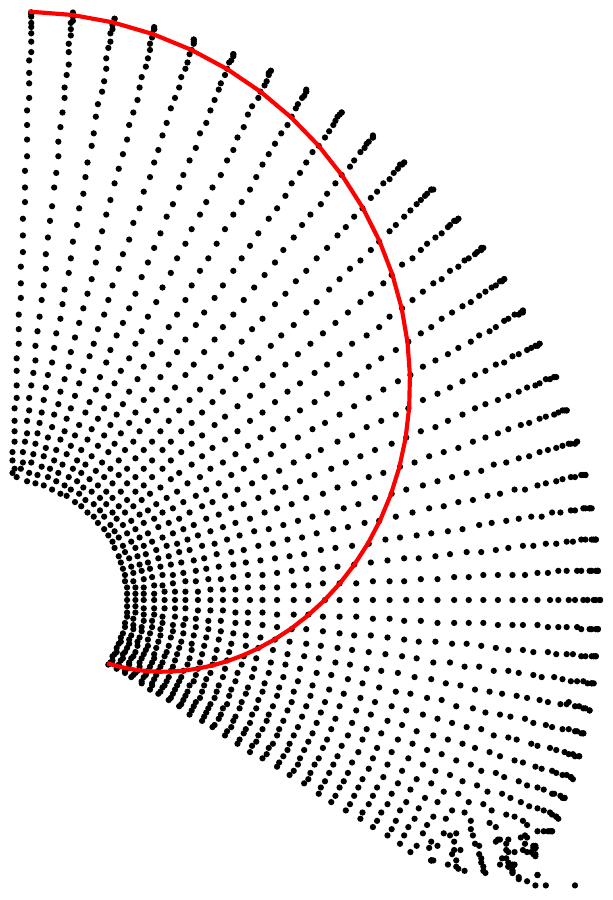}
 \caption{Left: Given curve $\F_0$ (red) and black points of the discrete minimal surface $\F_{m,n}$ obtained from our construction procedure in Section~\ref{secConstruction}. Right: Image of the discrete holomorphic map $G_{m,n}$ (black points) and curve of initial values $G_0$ (red). }\label{fig:Ex2}
\end{figure}

Using our construction procedure we have realised two examples in Figures~\ref{fig:Ex1} and~\ref{fig:Ex2}. In both examples, we have chosen a curve $\F_0$ (marked in red) on the Catenoid which is nowhere tangent to a curvature line. Furthermore, we have taken corresponding normals $\No_0$. The stereographic projection of the curve of normals, that is the trace of $G_0$, is displayed in red on the right of Figures~\ref{fig:Ex1} and~\ref{fig:Ex2} respectively. Additionally, the Figures show the images of the discrete holomorphic maps $G_{m,n}$ obtained by evolution according to steps~(iii)-(v). The discrete parameter space obtained as in Section~\ref{secDisHolo} from the curve $\gamma$ is 
\begin{itemize}
 \item a square lattice as $\gamma(t)=(1-i)t$ for the example in Figure~\ref{fig:Ex1} and 
 \item a rectangular lattice as $\gamma(t)=\frac{3}{2}(1-i)(\sin(\frac{t}{3})+i(1-\cos(\frac{t}{3})))$ for the example in Figure~\ref{fig:Ex2}.
\end{itemize}
For both examples we used the mesh size $\eps=0.1$. In Figure~\ref{fig:Ex2}, the parameter domain is bigger than the actual domain of convergence which results in divergent values for $G_{m,n}$ at one ``corner'' and correspondingly divergent values for the discrete minimal surface.

\section{Cross-ratio evolution: (discrete) mappings from Cauchy data}\label{secCR}
As detailed in Section~\ref{secDisHolo}, our construction is local.
Starting from suitable initial data 
for $G_{m,n}$ on a 'zig-zag'-curve in the parameter space $\Omega^\eps$, that is on the diagonal and on the first upper off-diagonal as indicated in Figure~\ref{fig:zigzag}, we are interested in a corresponding solution $G_{m,n}$ of the Cauchy problem to 
  \begin{align}
    \label{eq:devolve02}
    \frac{G_{m-1,n}-G_{m,n}}{G_{m,n}-G_{m,n+1}} \cdot 
\frac{G_{m,n+1}-G_{m-1,n+1}}{G_{m-1,n+1}-G_{m-1,n}} = 
\frac{(p_{m,n}-p_{m-1,n})^2}{(p_{m,n+1}-p_{m,n})^2}.
  \end{align}
  Our constructions and proof of Theorem~\ref{theo:ConvG} heavily relie on the fact, that we rewrite this  discrete elliptic equation as an initial value problem for a hyperbolic evolution equation.

In this section, we reformulate our problem and derive a discrete evolution equation for an auxialary function and its corresponding smooth counterpart. The solutions of these equations are further studied in Section~\ref{secApprox}. 

\subsection{Derivation of a discrete evolution equation}\label{secDeriv}
Inspired by~\cite[Sec.~5]{Ma05}, we first derive from the cross-ratio equation~\eqref{eq:devolve02} a discrete 
evolution equation for a suitable auxiliary function $F$ for which we then prove convergence to the corresponding smooth counterpart. To 
this end, we start by considering the quotients 
\begin{align}
\alpha_{m,n} &=\frac{G_{m,n+1}-G_{m,n}}{p_{m,n+1}-p_{m,n}}, & \label{eq:defalphabeta}
\beta_{m,n} &= \frac{G_{m,n}-G_{m-1,n}}{p_{m,n}-p_{m-1,n}},\\
Q_{m,n} &= \frac{\beta_{m,n}}{\alpha_{m,n}}= 
\frac{p_{m,n+1}-p_{m,n}}{p_{m,n}-p_{m-1,n}}\cdot 
\frac{G_{m,n}-G_{m-1,n}}{G_{m,n+1}-G_{m,n}}. && \label{eq:defQ}
\end{align}
At the moment we associate these values to vertices of the lattice $\Omega^\eps$, we could also think of these quantities to be defined on the centers of the edges and the rectangular 
faces of the lattice, that is $\Omega^\eps_*$, respectively. For notational 
convenience we will ignore this interpretation until the end of this section.

For further use we define the shift operators $\tau_1$ and $\tau_2$ which shift the indices in $m$- and $n$-direction respectively:
\begin{equation*}
 \tau_1 G_{m,n}= G_{m+1,n}\qquad \text{and}\qquad
\tau_2 G_{m,n}= G_{m,n+1}.
\end{equation*}
 The corresponding shifts in the negative directions are denoted as 
$\tau_{\overline{1}}$ and $\tau_{\overline{2}}$.

\begin{figure}[tb]
\begin{tikzpicture}[rotate=-45,scale=1.5]
\coordinate (A) at (-1, -1);
\coordinate (B) at (0,-1);
\coordinate (C) at (1,-1);
\coordinate (D) at (-1, 0);
\coordinate (E) at (0,0);
\coordinate (F) at (1,0);
\coordinate (G) at (-1,1);
\coordinate (H) at (0,1);
\coordinate (J) at (1,1);
\draw (A) --  (B) node[midway,below] {$\tau_{\overline{1}}\beta$};
\draw (B) -- (C) node[midway,left] {$\beta$};
\draw (C) -- (F) node[midway,right] {$\alpha$}; 
\draw (F) --(J) node[midway,right] {$\tau_2\alpha$}; 
\draw (J) --(G) --(A);
\draw (D)--(E);
\draw (E) -- (H) node[midway,above] {$\tau_{\overline{1}}\tau_2\alpha$};
\draw (E)--(F) node[midway,below] {$\tau_2\beta$};
\draw (B) -- (E) node[midway,above] {$\tau_{\overline{1}}\alpha$};
\end{tikzpicture}
\caption{Four adjacent rectangles and corresponding 
variables associated to edges (or to the left and lower vertices respectively).}\label{fig:foursquares}
\end{figure}
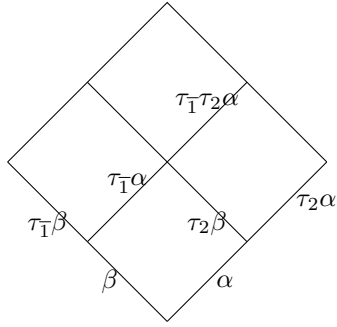
First we gather some more information in order to derive a (non-linear) evolution equation for~$Q$.
The cross-ratio condition~\eqref{eq:devolve02}
implies that $\frac{\beta\, \tau_2\beta}{\alpha\, \tau_{\overline{1}}\alpha}=1$, so
\begin{equation}\label{eq:Q2}
\tau_2\beta =\frac{1}{Q}\tau_{\overline{1}}\alpha.
\end{equation}
Furthermore, we have the closing condition for the edges
\begin{equation}\label{eqclose}
\delta_2p\; \alpha-\delta_{\overline{1}}p\; \tau_2\beta= 
-\delta_{\overline{1}}p\; \beta+\delta_2p\; \tau_{\overline{1}} \alpha,
\end{equation}
where $\delta_{\overline{1}}p\; =p-\tau_{\overline{1}} p$ and $\delta_2p\; =\tau_2 p-p$ denote the distances between neighboring points in the two directions of the rectangular lattice.
This implies in particular
\begin{align}
 &\delta_2p\;  - 
\delta_{\overline{1}}p\; \underbrace{\frac{\tau_2\beta}{\alpha}}_{= 
\frac{\tau_{\overline{1}}\alpha}{Q\alpha}}  = -\delta_{\overline{1}}p\;  Q+ 
\delta_2p\; \frac{\tau_{\overline{1}}\alpha}{\alpha} 
&\iff\qquad & \tau_{\overline{1}}\alpha= \frac{\delta_2p\;  +\delta_{\overline{1}}p\;  
Q}{\delta_2p\;  
+\frac{\delta_{\overline{1}}p\; }{Q}} \alpha. \label{eqalpha-}
\end{align}
Now we additionally consider the next quadrilateral on top, that is shifted by $\tau_2$, together with the corresponding quantities like $\tau_2 \alpha$ and 
$\tau_2 Q=\frac{\tau_2\beta}{\tau_2\alpha}$. We deduce from the closing 
condition~\eqref{eqclose} and the previous result that
\begin{align}
 &\delta_2p\; \frac{\alpha}{\tau_2\alpha}-
\delta_{\overline{1}}p\; \frac{\tau_2\beta}{\tau_2\alpha}= 
-\delta_{\overline{1}}p\; \frac{\beta}{\alpha} \frac{\alpha}{\tau_2\alpha} 
+\delta_2p\; \frac{\tau_{\overline{1}}\alpha}{\alpha} 
\frac{\alpha}{\tau_2\alpha}.\notag \\
\iff\qquad & \delta_2p\; \frac{\alpha}{\tau_2\alpha}- \delta_{\overline{1}}p\; \tau_2 Q 
=
\left(-\delta_{\overline{1}}p\; Q+ \delta_2p\; \frac{\delta_2p\;  -\delta_{\overline{1}}p\;  
Q}{\delta_2p\;  +\frac{\delta_{\overline{1}}p\; }{Q}}\right) 
\frac{\alpha}{\tau_2\alpha} \notag\\
\iff\qquad & \tau_2\alpha= \frac{(\delta_{\overline{1}}p\; Q+\delta_2p\; )(\delta_2p\;  
+\frac{\delta_{\overline{1}}p\; }{Q})- 
\delta_2p\; (\delta_2p\;  +\delta_{\overline{1}}p\;  Q)}{\delta_{\overline{1}}p\;  
\tau_2Q(\delta_2p\;  +\frac{\delta_{\overline{1}}p\; }{Q})} \alpha =
\frac{\delta_{\overline{1}}p\; Q+\delta_2p}{\tau_2Q(\delta_2p\;  Q
+\delta_{\overline{1}}p\;)}\alpha 
\label{eqalpha+}
\end{align}

Finally, we consider the next two quads on the left, that is in 
negative $\tau_1$-direction. Applying equations~\eqref{eqalpha+} 
and~\eqref{eqalpha-} in different orders, we get
\begin{align*}
 &\tau_{\overline{1}}(\tau_2\alpha)= 
\frac{\tau_{\overline{1}}\delta_{\overline{1}}p\;\tau_{\overline{1}}Q +\delta_2p\; }{\tau_{\overline{1}}(\tau_2 Q)(\delta_2p\;  \tau_{\overline{1}}Q
+\tau_{\overline{1}}\delta_{\overline{1}}p\; )}\tau_{\overline{1}}\alpha 
=\frac{(\tau_{\overline{1}}\delta_{\overline{1}}p\; \tau_{\overline{1}}Q +\delta_2p\; )}{\tau_2\tau_{\overline{1}}Q (\delta_2p\;  \tau_{\overline{1}}Q
+\tau_{\overline{1}}\delta_{\overline{1}}p\;)}\frac{(\delta_2p\;  
+\delta_{\overline{1}}p\;  Q)}{(\delta_2p\;  
+\frac{\delta_{\overline{1}}p\; }{Q})} \alpha \\
=\ &\tau_2(\tau_{\overline{1}}\alpha)=\frac{\tau_2\delta_2p\;  +\delta_{\overline{1}}p\;  
\tau_2Q}{\tau_2\delta_2p\;  +\frac{\delta_{\overline{1}}p\; }{\tau_2Q}}\tau_2\alpha
=\frac{(\tau_2\delta_2p\;  +\delta_{\overline{1}}p\;  
\tau_2Q)}{(\tau_2\delta_2p\;  +\frac{\delta_{\overline{1}}p\; }{\tau_2Q})} 
\frac{(\delta_{\overline{1}}p\; Q+\delta_2p\; )}{\tau_2Q(\delta_2p\; Q 
+\delta_{\overline{1}}p\; )}\alpha 
\end{align*}
This leads to an evolution equation for the function $Q$:
\begin{align*}
 \frac{\tau_2\tau_{\overline{1}}Q}{Q} &= \frac{\tau_2Q}{Q}
\frac{(\tau_{\overline{1}}\delta_{\overline{1}}p\;\tau_{\overline{1}}Q +\delta_2p\;)}{(\delta_2p\;  \tau_{\overline{1}}Q
+\tau_{\overline{1}}\delta_{\overline{1}}p\;)} \frac{(\delta_2p\;Q  
+\delta_{\overline{1}}p\;)}{(\delta_2p\; +\frac{\delta_{\overline{1}}p\; }{Q})} \frac{(\tau_2\delta_2p\;  
+\frac{\delta_{\overline{1}}p\; }{\tau_2Q})}{(\tau_2\delta_2p\;  +\delta_{\overline{1}}p\;  
\tau_2Q)}\\
& = \frac{(\tau_2\delta_2p\; \tau_2Q +\delta_{\overline{1}}p\; )}{(\tau_2\delta_2p\;  
+\delta_{\overline{1}}p\;  
\tau_2Q)} \frac{(\delta_2p\; 
+\tau_{\overline{1}}\delta_{\overline{1}}p\; \tau_{\overline{1}}Q)}{(\delta_2p\;  \tau_{\overline{1}}Q
+\tau_{\overline{1}}\delta_{\overline{1}}p\; )}.
\end{align*}

We make the ansatz $Q=\text{e}^{\eps F}$ for some function $F$.
Inserting the ansatz into the evolution equation for $Q$ yields 
\begin{align*}
 \frac{\text{e}^{\eps \tau_2\tau_{\overline{1}}F}}{\text{e}^{\eps F}} 
 =\frac{(\tau_2\delta_2p\; \text{e}^{\eps \tau_2F} 
+\delta_{\overline{1}}p\; )}{(\tau_2\delta_2p\;  +\delta_{\overline{1}}p\;  \text{e}^{\eps \tau_2F})} 
\frac{(\delta_2p\; +\tau_{\overline{1}}\delta_{\overline{1}}p\; \text{e}^{\eps 
\tau_{\overline{1}}F})}{(\delta_2p\;  \text{e}^{\eps \tau_{\overline{1}}F} 
+\tau_{\overline{1}}\delta_{\overline{1}}p\; )}.
\end{align*}

Taking the logarithm on both sides and dividing by $\eps$ we obtain
\begin{thm}
 Given a CR-mapping $G$ on a lattice $\Omega^\eps$, define $Q$ by~\eqref{eq:defQ}. Then the discrete mapping $F$ defined by $Q=\text{e}^{\eps F}$ satisfies the following evolution equation.
\begin{align}\label{eq:evolvF}
 \tau_2\tau_{\overline{1}}F- F=&\ M(\tau_2F-\tau_{\overline{1}}F) 
+2i\eps\;\Xi\; \frac{\tau_2F+\tau_{\overline{1}}F}{2} +\eps^2{\cal R}, \\
\end{align} 
where
\begin{align}
M &=-\frac{1}{2}\left(\frac{\tau_{\overline{1}}\delta_{\overline{1}}p-\delta_2p}{\tau_{\overline{1}} \delta_{\overline{1}}p+\delta_2p} +
\frac{\delta_{\overline{1}}p-\tau_2\delta_2p}{\delta_{\overline{1}}p+\tau_2\delta_2p }\right) 
=-\frac{ \delta_{\overline{1}}p\; \tau_{\overline{1}}\delta_{\overline{1}}p-\delta_2p\; \tau_2\delta_2p
}{(\tau_{\overline{1}}\delta_{\overline{1}}p +\delta_2p)(\delta_{\overline{1}}p+\tau_2\delta_2p)},\label{eqM}
\\[1.5ex]
\Xi&=\frac{1}{2i\eps}\left(\frac{ \tau_{\overline{1}}\delta_{\overline{1}}p-\delta_2p}{\tau_{\overline{1}} \delta_{\overline{1}}p +\delta_2p} -
\frac{\delta_{\overline{1}}p-\tau_2\delta_2p}{\delta_{\overline{1}}p +\tau_2\delta_2p }\right) 
=-\frac{\delta_{\overline{1}}p\; \delta_2p - \tau_{\overline{1}}\delta_{\overline{1}}p\;  
\tau_2\delta_2p
}{i\eps(\tau_{\overline{1}}\delta_{\overline{1}}p+\delta_2p\;)(\delta_{\overline{1}}p+\tau_2\delta_2p)},\label{eqXi} 
\\[1.5ex]
{\cal R}&=\frac{1}{\eps^3}\left( \log\left(
 \frac{\tau_2\delta_2p\; \text{e}^{\eps \tau_2F} 
+\delta_{\overline{1}}p\; }{\tau_2\delta_2p\;  +\delta_{\overline{1}}p\;  
\text{e}^{\eps \tau_2F}}\right)+ 
\eps\frac{\delta_{\overline{1}}p-\tau_2\delta_2p}{\delta_{\overline{1}}p+\tau_2\delta_2p } 
\tau_2F \right. \notag \\
&\qquad\quad \left. -\log\left(\frac{\delta_2p\;  
\text{e}^{\eps \tau_{\overline{1}}F} 
+\tau_{\overline{1}}\delta_{\overline{1}}p\; }{\delta_2p\; 
+\tau_{\overline{1}}\delta_{\overline{1}}p\; \text{e}^{\eps 
\tau_{\overline{1}}F}}\right) 
-\eps \frac{\tau_{\overline{1}}\delta_{\overline{1}}p-\delta_2p}{\tau_{\overline{1}} \delta_{\overline{1}}p+\delta_2p} \tau_{\overline{1}}F\right). \notag
\end{align}
\end{thm}

The remainder $\cal R$ may be expressed as
${\cal R}= R(\tau_2\delta_2p\;,\delta_{\overline{1}}p\;,\tau_2F;\eps)-R(\delta_2p\;,\tau_{\overline{1}}\delta_{\overline{1}}p\;,\tau_{\overline{1}}F;\eps)$,
where 
\[R(a,b,F;\eps)=\frac{1}{\eps^3}\left( \log\left(\frac{a\text{e}^{\eps F}+b}{a+b\text{e}^{\eps F}}\right) +\eps \frac{b-a}{a+b} F\right).\]
Note that in our case $a\in i\R$ and $b\in\R$.

Comparing~\eqref{eq:evolvF} to the structure of the evolution equation considered in~\cite[Sec.~5]{Ma05}, there is a similar structure for the linear terms, but our constants $M$ and $\Theta$ depend on the position $p$ in the parameter lattice $\Omega^\eps$. Additionally there is the remainder $\cal R$, which  fortunately turns out to be bounded, but nonetheless burdens the estimates in Section~\ref{secConvF}.

At this point, we re-locate the values of $F$ 
(which arose from the values of $Q$). Note that the midpoints of the quad in the lattice $\Omega^\eps$ build another rectangular lattice 
$\Omega^\eps_*\subset\Omega$.
Instead of the lower right vertex of a 
rectangle, we now associate the values of $F$ to the center: $F_{m-\halb, n+\halb}$, that is to a vertex of $\Omega^\eps_*$. Similarly, we associate 
the quantities $M$ and $\Xi$ not to the lower right vertex of the configuration 
of four incident quads, but to their common vertex $p_{m-1,n+1}$, that is 
$M_{m-1,n+1}$ and $\Xi_{m-1,n+1}$. With this notation, we obtain the discrete 
evolution equation

\begin{multline}\label{eq:devolveFmn}
 F_{m-1-\halb,n+1+\halb}- F_{m-\halb,n+\halb}= M_{m-1,n+1}
(F_{m-\halb,n+1+\halb} -F_{m-1-\halb,n+\halb})  \\
 \shoveright{+2i\eps\cdot \Xi_{m-1,n+1} \cdot\frac{F_{m-\halb,n+1+\halb} 
+F_{m-1-\halb,n+\halb}}{2} \quad}
\\
 +\eps^2(R_{m-\halb,n+1+\halb}(F;\eps) 
-R_{m-1-\halb,n+\halb}(F;\eps)),
\end{multline}
where $R_{m+\halb,n+\halb}(F;\eps) = R((p_{m,n+1}-p_{m,n}), (p_{m+1,n}-p_{m,n}), F_{m+\halb,n+\halb};\eps)$.

In the following, we further study the discrete function $F$ as solution of the discrete evolution equation~\eqref{eq:devolveFmn} and prove its convergence. To this end, we first show that the consistency of the discrete equation and then adapt ideas from~\cite{Ma05} in order to prove $C^0$- and $C^\infty$-convergence.
Finally, the convergence of the functions $\alpha$, $\beta$ and $G$ will be deduced in Sections~\ref{secConvalphabeta} and~\ref{secConst2}.

\subsection{Consistency of the discrete evolution equation~\eqref{eq:devolveFmn}}
A straightforward Taylor expansion suggests that the smooth function which 
corresponds to 
$F$ is $-\frac{1}{2}\frac{g''}{g'}(u'+iv')$. In the 
following, we will derive the corresponding smooth evolution equation and show 
its consistency with the discrete equation.

A key concept in the proof is to work with analytic extensions of the quantities $u$ and $v$. Therefore, we introduce another class of domains.
For $0<r_0<1$ sufficiently small, we define
the double cone
\begin{align}
  \label{eq:32}
  \Diamond_{r_0} := \{(\xi,\eta)\in\C\times\R\,|\,|\xi|+|\eta|\le r_0\},
\end{align}
where $\xi$ is the extension of the real parameter $t$, see Figure~\ref{fig:mesh}. 
Recall that the function $g$ is defined on a domain in the $u$-$v$-plane and we use the coordinate transformation $p(\xi,\eta)=u(\xi-\eta)+iv(\xi+\eta)$ as extension of~\eqref{eq:p1}. We define a function $f:\Diamond_{r_0} \to\C$ by
\begin{align}
  \label{eq:42}
  f(\xi,\eta) &=-\frac{1}{2} \partial_\xi \log g'\circ p(\xi,\eta)= -\frac{1}{2} 
\frac{g''\big(u(\xi-\eta)+iv(\xi+\eta)\big)}{
g'\big(u(\xi-\eta)+iv(\xi+\eta)\big)} (u'(\xi-\eta)+iv'(\xi+\eta)). 
\end{align}
This function satisfies
\begin{align*}
  \partial_\xi f = -\frac{1}{2} (u'+iv')^2\,(\frac{g''}{g'})' 
-\frac{1}{2}\frac{g''}{g'}(u''+iv'') ,
  \quad
  \partial_\eta f = \frac{1}{2}(u'-iv')(u'+iv')\,(\frac{g''}{g'})' 
-\frac{1}{2}\frac{g''}{g'}(-u''+iv'') ,
\end{align*}
and therefore, we obtain the corresponding smooth ``evolution'' equation
\begin{align}
  \label{eq:evolvef}
  &\partial_\eta f = \Theta\cdot\partial_\xi f + 2i\cdot\Xi\cdot f,
  \quad\\
\text{where }\quad  &\Theta(\xi,\eta) = \frac{\partial_\eta p(\xi,\eta)}{\partial_\xi p(\xi,\eta)}=
-\frac{u'(\xi-\eta)-iv'(\xi+\eta)}{u'(\xi-\eta)+iv'(\xi+\eta)} \quad\\
\text{and }\quad&\Xi(\xi,\eta) = \frac{1}{2i}\partial_\xi \Theta(\xi,\eta) =
-\frac{u''(\xi-\eta) v'(\xi+\eta) -v''(\xi+\eta) u'(\xi-\eta)}{(u'(\xi-\eta) 
+iv'(\xi+\eta))^2}.
\end{align}
Notice that $|\Theta(\xi,\eta)|=1$ for real arguments $\xi$ and $\eta$. Furthermore, \eqref{eq:evolvef} represents in fact the Cauchy-Riemann equations, which takes into account our coordinate transformation $p$. For $u=id$ and $v=id$ we obtain the usual form.

Instead of dealing with $F$ directly, we consider a suitable semi-discrete function $F^\eps:\Diamond^\eps_{r_0}\to\C$  on the time-discretized double cone
\begin{align}
  \label{eq:82}
  \Diamond^\eps_{r_0}:= \{(\xi,\eta)\in\Diamond_{r_0}\,| \,\eta\in\epss\Z\}.
\end{align}
We can also think of this domain as a 1-parameter family of shifted lattices with the complex parameter $\xi$.
For further use, we define the {\em discrete partial derivatives} on 
$\Diamond^\eps_{r_0}$ by
\begin{align}
 \delta_\eta H(\xi,\eta)= \frac{1}{\eps} (H(\xi,\eta+\epss)-H(\xi,\eta-\epss))
\quad \text{and}\quad
 \delta_\xi H(\xi,\eta)= \frac{1}{\eps} (H(\xi+\epss,\eta)-H(\xi-\epss,\eta))
\end{align}
as well as the {\em mean value operator} $I_\xi$ via
\begin{align}
 I_\xi H(\xi,\eta)= \frac{1}{2} (H(\xi+\epss,\eta)+H(\xi-\epss,\eta))
\end{align}
and the ratio
\begin{align}\label{eq:defThetaeps}
 \Theta^\eps(\xi,\eta) &=\frac{\delta_\eta p(\xi,\eta)}{\delta_\xi p(\xi,\eta)}
 =-\frac{u(\xi-\eta+\epss)-u(\xi-\eta-\epss) 
-i(v(\xi+\eta+\epss)-v(\xi+\eta-\epss))}{u(\xi-\eta+\epss)-u(\xi-\eta-\epss) 
+i(v(\xi+\eta+\epss)-v(\xi+\eta-\epss))}
\end{align}

In order to obtain a discrete evolution equation for $F^\eps$, we introduce the functions
\begin{align}
  M^\eps(\xi,\eta) 
  &= I_\xi \Theta^\eps(\xi,\eta) =\frac{1}{2}\left( \Theta^\eps(\xi+\epss,\eta) 
+\Theta^\eps(\xi-\epss,\eta) \right), \\
 \Xi^\eps(\xi,\eta) &=\frac{1}{2i}\delta_\xi \Theta^\eps(\xi,\eta) =\frac{1}{2i\eps}\left( \Theta^\eps(\xi+\epss,\eta) 
-\Theta^\eps(\xi-\epss,\eta) \right),  
\end{align}
and
\begin{multline}\label{eq:defR}
 R(F^\eps;\xi\pm\epss,\eta,\eps)= \\
  \frac{1}{\eps^3} 
 \Lo \left(u(\xi-\eta+\epss \pm\epss) -u(\xi-\eta-\epss \pm\epss), v(\xi+\eta+\epss \pm\epss) -v(\xi+\eta-\epss \pm\epss), \eps F^\eps(\xi\pm \epss,\eta)\right),\\
\end{multline}
where
\begin{equation*}
\Lo(\ell_a, \ell_b, {\cal H})= \log\left(\frac{i\ell_b\text{e}^{{\cal H}} +\ell_a}{i\ell_b 
+\ell_a \text{e}^{{\cal H}}}\right)+ {\cal H} 
\frac{\ell_a-i\ell_b}{\ell_a+i\ell_b}.
\end{equation*}

Equipped with this notation, we can apply analogous reasoning as in the previous section starting from the curves $\gamma_{t_0}(t)=u(t_0+t)+iv(t_0+t)$ in order to define the values of $F^\eps$ on the corresponding shifted lattices. This can be summarized in the following relation:
\begin{lemma}\label{lem:devolve2}
 The semi-discrete function $F^\eps$ satisfies the discrete evolution equation
\begin{multline}
  \label{eq:devolve2}
  \delta_\eta F^\eps(\xi,\eta)
  = M^\eps(\xi,\eta)\delta_\xi F^\eps(\xi,\eta)
 + 2i\cdot \Xi^\eps(\xi,\eta) \cdot I_\xi F^\eps(\xi,\eta) \\ 
+\eps^2(R(F^\eps;\xi+\epss,\eta,\eps) -R(F^\eps;\xi-\epss,\eta,\eps)).
\end{multline}
\end{lemma}

Equation~\eqref{eq:devolve2} is compatible with \eqref{eq:devolveFmn} in the following sense:
if a function $F^\eps:\Diamond^\eps_{r_0}\to\C$ satisfies \eqref{eq:devolve2}, 
then the ``projection'' of its values given by
\begin{align}\label{eq:Fproj}
  F_{m-\halb,n+\halb} = F^\eps\big((n+m)\epss,(n-m+1)\epss\big),
\end{align}
satisfies \eqref{eq:devolveFmn}.

\begin{rmk}\label{rem:TaylorF}
 For further use, we consider for positive parameters $\ell_a,\ell_b$ the function $\Lo(\ell_a, \ell_b, \cdot)$,
 \begin{equation}\label{eq:defRest}
{\cal H}\mapsto \log\left(\frac{i\ell_b\text{e}^{{\cal H}} +\ell_a}{i\ell_b +\ell_a \text{e}^{{\cal H}}}\right)+ {\cal H} 
\frac{\ell_a-i\ell_b}{\ell_a+i\ell_b}
\end{equation}
on the domain $\R\times[-i\pi/4,i\pi/4]$ in the complex plane. Lateron, the interval $[-A,A] \times[-i\pi/4,i\pi/4]$ for a suitable $A>0$ will be sufficient. Observe that $\Lo(\ell_a, \ell_b, \cal H)$ is
is analytic in ${\cal H}$ in a neighborhood of zero and odd. Furthermore, for small $\cal H$ a
Taylor expansion (using a computer algebra program) gives
\begin{equation*}
\log\left(\frac{i\ell_b\text{e}^{{\cal H}} +\ell_a}{i\ell_b +\ell_a
\text{e}^{{\cal H}}}\right) +{\cal H} 
\frac{\ell_a-i\ell_b}{\ell_a+i\ell_b} =\Or({\cal H}^3).
\end{equation*}
This shows in particular, that $\widehat{\Lo}(\ell_a, \ell_b, \cal H):= \frac{1}{\cal H^3}\Lo(\ell_a, \ell_b, \cal H)$ as function in $\cal H$ has a solvable singularity at the origin. Therefore by~\eqref{eq:defR}, for uniformly bounded $F^\eps$ and $\eps$ small enough the ``remainder'' 
term $\eps^2(R(F^\eps;\xi+\epss,\eta,\eps) -R(F^\eps;\xi-\epss,\eta,\eps)) 
=\Or(\eps^2)$ is small.
%
 In Section~\ref{secExist} we will show that for suitable initial data as detailed in~\eqref{eq:init1}--\eqref{eq:init2} or~\eqref{eq:init1b}--\eqref{eq:init2b}, there exists a uniform bound on $F^\eps$, see also~\eqref{eq:estFn2}. 
\end{rmk}

\begin{lemma}\label{lemConsist}
Let $F^\eps$ be a solution of~\eqref{eq:devolve2} and let $f$ be a solution 
of~\eqref{eq:evolvef}. Let $\Delta F=F^\eps-f$. Then
\begin{align*} 
 \delta_\eta \Delta F(\xi,\eta)
  = &M^\eps(\xi,\eta)\delta_\xi \Delta F(\xi,\eta)
 + 2i\cdot \Xi^\eps(\xi,\eta) \cdot I_\xi \Delta F(\xi,\eta) \\
 &+\eps^2(R(\Delta F;\xi+\epss,\eta,\eps) -R(\Delta F;\xi-\epss,\eta,\eps))
+\eps {\cal S}(f,F;\xi,\eta,\eps),
\end{align*}
where $|{\cal S}|\leq C$ uniformly and $C$ depends on $F^\eps,f,u,v$, but not on $\eps$.
\end{lemma}
\begin{proof}
 By Taylor expansions of $\delta_\xi f$, $\delta_\eta f$, 
$M^\eps =-\frac{(u')^2+(v')^2}{(iv'+u')^2} +\Or(\eps^2)$, and 
$\Xi^\eps =\frac{2i(u'v''-v'u'')}{(iv'+u')^2}  +\Or(\eps^2)$,
we easily deduce from~\eqref{eq:evolvef} that
\[ \delta_\eta f=M^\eps\cdot \delta_\xi f+ 2i\;\Xi^\eps\cdot f 
+\Or(\eps^2).\]
Remark~\ref{rem:TaylorF} implies that $R(\Delta F;\xi+\epss,\eta,\eps) 
-R(\Delta F;\xi-\epss,\eta,\eps) =\Or(1)$ and $R(F^\eps;\xi+\epss,\eta,\eps) 
-R(F^\eps;\xi-\epss,\eta,\eps) =\Or(1)$. The claim follows with Lemma~\ref{lem:devolve2}.
\end{proof}

\section{$C^\infty$-convergence of $F^\eps$ to $f$}\label{secConvF}
In this section we prove that with suitably chosen initial conditions 
exists and approximates $f$.

\begin{thm}\label{theoConvF}
For a given function $f$ in~\eqref{eq:42} the solution of the discrete Cauchy problem consisting of the evolution 
equation~\eqref{eq:devolve2} for $F^\eps$ and $(\xi,\eta)\in \Diamond^\eps_r$, where $0<r<r_0$, together  with the initial values
\begin{align}
 F^\eps(\xi,0)&= f(\xi,0), \label{eq:init1}\\
 F^\eps(\xi,\epss)&= f(\xi,0) + \epss(\Theta(\xi,0)\cdot \partial_\xi 
f(\xi,0) +2i\;\Xi(\xi,0) f(\xi,0)). \label{eq:init2}
\end{align}
exists and approximates $f$ with all its derivatives.
\end{thm}

The methods of the proofs extend ideas from~\cite[Section 3]{Ma05} to our case. This technical adaptions are mostly due to the fact that $M^\eps$ and $\Theta^\eps$ are not constant and we have to deal with an additional term $\eps^2(R(F^\eps;\xi+\epss,\eta,\eps) -R(F^\eps;\xi-\epss,\eta,\eps))$ with $R$ given by~\eqref{eq:defR}.

\subsection{Some useful norms and their basis properties}
Adapting ideas from~\cite{Ma05}, we consider for smooth functions $h$ defined on $\{|\xi|<\rho_0\}$ and $0<\rho<\rho_0$ the norms
\begin{align*}
  \|h\|_\rho = \sum_{k=0}^\infty 
\frac{(B\rho)^k}{k!}\sup_{|\xi|\le\rho}|h^{(k)}(\xi)|,
\end{align*}
where $B\ge1$ is a constant that depends on $u$ and $v$.
We choose $B\geq 5\sup_{\Diamond_{r_0}}|\Theta^\eps|$ for $\eps$ small enough.
That norm has the advantage that for $\rho+\eps<\rho_0$ we have
\begin{align}
  \label{eq:magic}
  \|h\|_\rho + B\eps\|h'\|_\rho \le \|h\|_{\rho+\eps},
\end{align}
since
\begin{align*}
   B\|h'\|_\rho = \sum_{k=1}^\infty 
\frac{B^k\rho^{k-1}}{(k-1)!}\sup_{|\xi|\le\rho}|h^{(k)}(\xi)|
  = \sum_{k=0}^\infty \frac 
k{\rho}\frac{(B\rho)^k}{k!}\sup_{|\xi|\le\rho}|h^{(k)}(\xi)|,
\end{align*}
and since
\begin{align*}
  \big(B(\rho+\eps)\big)^k \ge \left(1+\eps\frac k\rho\right)(B\rho)^k.
\end{align*}

Similarly, we can define norms on discrete functions $W:{\cal I}_n^\eps\to\C$, 
where
\begin{equation*}
 {\cal I}_n^\eps=\begin{cases} [-\frac{m}{2}\eps, \frac{m}{2}\eps]\cap \eps\Z & 
\text{if } m \text{ is even}, \\
[-\frac{m}{2}\eps, \frac{m}{2}\eps]\cap (\eps\Z+\epss) &\text{if } m \text{ is 
odd}.
\end{cases}
\end{equation*}
For $\rho>0$ set
\begin{equation}
  \|W\|_\rho = \sum_{k=0}^n
\frac{(B\rho)^k}{k!} \max_{\xi\in {\cal I}_{n-k}^\eps} 
|\delta_\xi^{(k)}W(\xi)|.
\end{equation}

With analogous proofs as for~\cite[Lemma~3.1]{Ma05}, this norm has 
the following properties:
\begin{lemma}\label{lemPropNorm}
\begin{enumerate}[(1)]
 \item For $0\leq \rho'\leq \rho$ one has $ \|W\|_{\rho'}\leq \|W\|_\rho$.
\item \emph{Absolute bound}: $|W(x)|\leq \|W\|_\rho$ for all $x\in{\cal 
I}_n^\eps$.
\item \emph{Submultiplicity}: $\|W_1 W_2\|_\rho \leq\|W_1\|_\rho \|W_2\|_\rho$
\item \emph{Discrete Cauchy estimate}: For $\theta>0$ we have $\|W\|_\rho + 
B\theta \|\delta_xW\|_\rho \le \|W\|_{\rho+\theta}$.
\item \emph{Restriction estimate}: If $w:\{|\xi|\leq \rho'\}\to \C$ is analytic and if
$W^\eps$ is its restriction to ${\cal I}_n^\eps \subset \{|\xi|\leq \rho'\}$, 
then for and $\rho<\rho'/B$ we have
\begin{equation*}
 \|W\|_\rho \leq \left(1-\frac{B\rho}{\rho'}\right)^{-1} \sup_{|\xi|\le\rho'}|w(\xi)|
\end{equation*}
\item \emph{Analyticity estimate}: Let functions $W^\eps: {\cal 
I}_{n^\eps}^\eps\to \C$ be given with $n^\eps \eps\geq s\geq 0$ and assume 
that $\|W^\eps\|_\rho\leq C$ for a sequence $\eps\to 0$.
Denote by $L[\delta_\xi^k W^\eps]$ the continuous function which interpolates 
linearly the values of $\delta_\xi^k W^\eps$.
Then there exists an analytic $w:[-s,s]\to\C$ such that $L[\delta_\xi^k 
W^{\eps(k)}]$ converges uniformly to $\partial_\xi^k w$ for each $k\geq 0$ and 
a suitable subsequence $\eps(k)\to 0$ of~$\eps$.
Moreover, $w$ possesses a complex extension to $\{\xi\in\C : 
\text{dist}(\xi,[-s,s])\leq B\rho \}$ which is bounded by~$C$.
\end{enumerate}
\end{lemma}

For multi-component functions $\sigma=(\sigma_1,\dots,\sigma_p) : 
{\cal I}_n^\eps \to\C^p$ we set $\|\sigma\|_\rho= \max_{k=1,\dots,p}\|\sigma_k\|_\rho$.
Note that also the following lemma still holds for our submultiplicative norms.

\begin{lemma}[{\cite[Lemma~3.2]{Ma05}}]\label{lemNorm2}
Let the analytic function ${\cal F} : \{|s|\leq A\}^p\ \to \C$ satisfy
\begin{equation*}
 |{\cal F}(s_1,\dots,s_p)| \leq {\cal C}(|s_1|,\dots,|s_p|)
\end{equation*}
for all $(s_1,\dots,s_p)\in \{|s|\leq A\}^p$, with some function ${\cal C} 
\geq 0$ that is non-decreasing in each of its
arguments. Then for each $\gamma > 1$ and every discrete function $\sigma : 
{\cal I}_n^\eps \to\C^p$ with $\gamma\cdot \max_{k=1,\dots,p}\|\sigma_k\|_\rho \leq 
A$, the composition ${\cal F}(\sigma)$ is well defined on ${\cal I}_n^\eps$ and
\begin{equation}
\|{\cal F}(\sigma)\|_\rho \leq \Gamma {\cal C}(\gamma \|\sigma_1\|_\rho, \dots, 
\gamma \|\sigma_1\|_\rho).
\end{equation}
The constant $\Gamma$ depends on $\gamma$ but not on the submultiplicative norm 
$\|\cdot\|_\rho$.
\end{lemma}
As shown in~\cite{Ma05} this also implies the two particular cases
\begin{align}
 \|{\cal F}(\sigma)\|_\rho &\leq \Gamma\cdot \sup_{|s_1|,\dots,|s_p| \leq A} 
|{\cal F}(s_1,\dots,s_p)|, \label{eq:estf}\\
 \|{\cal F}(\sigma^{(1)})- {\cal F}(\sigma^{(2)})\|_\rho &\leq \Gamma \cdot
\sup_{|s_1|,\dots,|s_p| 
\leq A} |{\cal F}'(s_1,\dots,s_p)| \|\sigma^{(1)}- \sigma^{(2)}\|_\rho. 
\label{eq:estdiff}
\end{align}

We will apply these estimates to the analytic function defined 
in~\eqref{eq:defRest}.

\subsection{Existence of a continuous solution of~\eqref{eq:evolvef}}\label{secExist}

Consider the discrete Cauchy problem consisting of the evolution 
equation~\eqref{eq:devolve2} for $F^\eps$ and $(\xi,\eta)\in \Diamond^\eps_r$, where $0<r<r_0$, 
together  with the initial values~\eqref{eq:init1} and~\eqref{eq:init2}.

\begin{rmk}\label{remAltIni}
Alternatively, we can use the following more symmetric choice as initial values
\begin{align}
 F^\eps(\xi,0)&= f(\xi,0) -\epsss(\Theta(\xi,0)\cdot \partial_\xi 
f(\xi,0) +2i\;\Xi(\xi,0) f(\xi,0)), \label{eq:init1b}\\
 F^\eps(\xi,\epss)&= f(\xi,0) + \epsss\left(\Theta(\xi,0)\cdot \partial_\xi 
f(\xi,0) +2i\;\Xi(\xi,0) f(\xi,0)\right). \label{eq:init2b}
\end{align}
\end{rmk}

\begin{lemma}\label{lemFeps}
 For a suitable $0<r< r_0/(5B)$ and $\eps$ small 
enough, discrete solutions to the above Cauchy problem exist on 
$\Diamond^\eps_r$ and a limiting function $f$ can be defined and solves~\eqref{eq:evolvef}.
\end{lemma}
\begin{proof}
Let $F_n^\eps(\xi)$ be 
the restriction of the solution $F^\eps$ of~\eqref{eq:init1}--\eqref{eq:init2} or~\eqref{eq:init1b}--\eqref{eq:init2b} to $\eta=n\epss$ for $|n|\epss\leq r$, so 
$F_n^\eps(\xi)= F^\eps(\xi, n\epss)$, which is defined on ${\cal I}_m^\eps$, 
where $m$ is the largest integer with $m\epss\leq r-|n|\epss$.
We assume $0<\eps<r$ and further reduce it in the following.
We first show that
\begin{align}
&& {\cal D}_n^\eps &:= c_n (\|F_{n+1}^\eps\|_{\mu_n} +\|F_n^\eps\|_{\mu_n}) 
\leq A\left(\frac{1}{2}+\frac{|n|}{2N}\right), \label{eq:Dn} \\
\text{where} && \mu_n&:=4r-|n|\eps, \qquad N\epss \leq r, \qquad
c_n :=\left(1+ \frac{2r C_\Xi}{N}\right)^{N-n}\geq c_{n+1},  \label{eq:mun_cn}\\
\text{ and}&&
C_\Xi:&=\frac{r_0}{Br}\sup_{\Diamond_{r_0}} |2\;\Xi^\eps| \geq \|2i\cdot 
\Xi_n^\eps\|_{\mu_n} \quad \text{by Lemma~\ref{lemPropNorm}~(5)} . \notag
\end{align}

We proceed by induction over $n$. Using the analyticity of $f,u,v$ and thus of 
$\Theta$ and $\Xi$ on $D_{r_0}\supset [-5Br,5Br]$, we see with property~(5) of Lemma~\ref{lemPropNorm} that $c_0\| 
F_0^\eps\|_{\mu_0}\leq A/4$ and $c_0\|F_1^\eps\|_{\mu_0}\leq A/4$ for some 
suitable constant $A>0$. Without loss of generality, we additionally assume  
that $A>\max\{2\text{e}^{2rC_\Xi},r\}$. 

Recall that by Remark~\ref{rem:TaylorF} the function 
\begin{equation}
{\cal H}\mapsto \frac{1}{{\cal H}^3} 
\left[\log\left(\frac{\ell_bi\text{e}^{{\cal H}} +\ell_a}{\ell_b 
-\ell_a i\text{e}^{{\cal H}}}\right)- i\frac{\pi}{2} + {\cal H} 
\frac{\ell_a-i\ell_b}{\ell_a+i\ell_b}\right]
\end{equation}
is analytic in ${\cal H}$ in a neighborhood of zero with a removable singularity at ${\cal H}=0$. So, for $|{\cal H}|\leq \pi/4$ there 
exists a bound which depends only on $\ell_a/\ell_b$. As this 
dependence is smooth, we deduce that for
\begin{align*}
&\ell_a^\eps = \frac{1}{\eps}(u(\xi-\eta+\epss ) -u(\xi-\eta-\epss 
))\in [\inf u', \sup u'] \text{ and } \\
 & \ell_b^\eps = \frac{1}{\eps}(v(\xi+\eta+\epss) -v(\xi+\eta-\epss ))\in [\inf v', \sup v']
\end{align*}
there exists an upper bound ${\cal R}$ which is 
independent of $\ell_a/\ell_b$. 
As these estimates are independent of $\eps$, we choose $\eps<1/(8A^4\Gamma{\cal R})$ with $\Gamma$ from Lemma~\ref{lemNorm2} 

Now suppose that~\eqref{eq:Dn} holds for some $n\geq 0$. Then with~\eqref{eq:devolveFmn} and~\eqref{eq:devolve2}
\begin{align*}
 {\cal D}_{n+1}^\eps & = c_{n+1}\|F_{n+1}^\eps\|_{\mu_{n+1}} \\
&\ +c_{n+1}\|F_n^\eps +\eps M_{n+1}^\eps \delta_\xi F_{n+1}^\eps + 
2i\eps \Xi_{n+1}^\eps I_\xi F_{n+1}^\eps + \eps^3(R_{n+1}(F_{n+1}^\eps; 
\cdot+\epss)- R_{n+1}(F_{n+1}^\eps; \cdot-\epss))\|_{\mu_{n+1}} \\
&\leq c_{n+1}(\|F_{n+1}^\eps\|_{\mu_{n+1}} 
+\|F_n^\eps\|_{\mu_{n+1}} +\eps \|\Theta_{n+1}^\eps\|_{\mu_{n}} \|\delta_\xi 
F_{n+1}^\eps \|_{\mu_{n+1}} + 
\eps C_\Xi\|F_{n+1}^\eps\|_{\mu_{n+1}} \\
&\qquad +\eps^3\|R_{n+1}(F_{n+1}^\eps; 
\cdot+\epss)- R_{n+1}(F_{n+1}^\eps; \cdot-\epss)\|_{\mu_{n+1}}).
\end{align*}
As $B> \|\Theta_{n+1}^\eps\|_{\mu_{n+1}}$, we obtain by~\eqref{eq:magic} (or 
property~(4) of Lemma~\ref{lemPropNorm})
\begin{equation*}
 \|F_{n+1}^\eps\|_{\mu_{n+1}} +\eps \|\Theta_{n+1}^\eps\|_{\mu_{n}} 
\|\delta_\xi F_{n+1}^\eps \|_{\mu_{n+1}}  \leq   
\|F_{n+1}^\eps\|_{\mu_{n+1}+\eps}.
\end{equation*}
Applying property~(1) of Lemma~\ref{lemPropNorm}, \eqref{eq:mun_cn} 
and~\eqref{eq:estf}, we obtain
\begin{align*}
 {\cal D}_{n+1}^\eps & \leq c_{n}\|F_n^\eps\|_{\mu_{n}} 
+ \underbrace{c_{n+1}(1+\eps C_\Xi)}_{\leq c_n} \|F_{n+1}^\eps\|_{\mu_{n}} 
+\underbrace{2\eps c_{n+1}A^3}_{\leq 2A^4\eps}\cdot \underbrace{\eps^2 
\Gamma{\cal R}}_{\leq 1/(8A^4)} \\
&\leq  {\cal D}_{n}^\eps +\frac{A}{2N}.
\end{align*}
This shows the claim for small enough $\eps>0$.

Estimate~\eqref{eq:Dn} shows in fact that
\begin{equation}\label{eq:estFn}
 \|F_n^\eps\|_{\mu_n}\leq A.
\end{equation}
This implies in particular that discrete solutions for the above Cauchy problem exist on $\Diamond^\eps_r$ and for all appropriate $n$ we have
\begin{align}
|F_n^\eps|&\leq A, \label{eq:estFn2} \\
 |\delta_\xi F_n^\eps| &\leq \frac{1}{B\mu_n} \|F_n^\eps\|_{\mu_n} \leq 
\frac{A}{2rB}, \label{eq:estdeltaxi} \\
|\delta_\eta F_n^\eps| &\leq |M_{n+1}^\eps||\delta_\xi F_n^\eps| + 
 |2i\cdot\Xi_{n+1}^\eps||I_\xi F_{n+1}^\eps| + \eps^2|R_{n+1}(F_{n+1}^\eps; 
\cdot+\epss)- R_{n+1}(F_{n+1}^\eps; \cdot-\epss)| \notag \\
&\leq (\frac{1}{2r}+C_\Xi) A +\eps^2\Gamma{\cal R}, \label{eq:estdeltaeta} \\
 |\delta_\xi^2 F_n^\eps| &\leq \frac{2}{B^2\mu_n^2} \|F_n^\eps\|_{\mu_n} \leq 
\frac{A}{2r^2B^2}.\label{eq:estdeltaxi2}
\end{align}

For every $\eps$ let $F^\eps_{even}$ be the restriction of $F^\eps$ to points 
$(\xi,\eta)\in\Diamond^\eps_r$ with ``even'' second coordinate 
$\eta=(2k)\epss$. This allows us to define a family of continuous functions 
$\hat{F}^\eps_{even}$ which are $\xi$-$\eta$-linear interpolations of 
$F^\eps_{even}$ on $\Diamond_r$. Estimates~\eqref{eq:estdeltaxi} 
and~\eqref{eq:estdeltaeta} show that this family is equicontinuous. Thus by the 
Arzel\`a-Ascoli theorem, there is a sequence $\eps'\to 0$ such that 
$\hat{F}^{\eps'}_{even}$ converges uniformly on $\Diamond_r$ to a continuous 
limit $f_e$. Using~\eqref{eq:estdeltaxi2}, we may choose the sequence $\eps'$ 
such that $\delta_\xi \hat{F}_{even}^{\eps'}$ converges uniformly to 
$\partial_\xi f_e$.
The same procedure may be applied to the restriction of $F^\eps$ to points 
$(\xi,\eta)\in\Diamond^\eps_r$ with ``odd'' second coordinate 
$\eta=(2k+1)\epss$. Passing to a suitable subsequence $\eps''$ of $\eps'$, we 
obtain that $\hat{F}^{\eps''}_{odd}\to f_o$ and $\delta_\xi 
\hat{F}_{odd}^{\eps''}\to \partial_\xi f_o$.
Observe, that $\Theta^{\eps"}\to \Theta$, $\Xi^{\eps"}\to \Xi$, and  
$(\eps'')^2 R(\hat{F}^{\eps"};\xi\pm\epss,\eta,\eps'')\to 0$.

As $F^\eps$ solves~\eqref{eq:devolve2}, we have for all 
$(\xi,\eta), (\xi,\tilde{\eta})\in\Diamond_r$ that
\begin{align*}
 \hat{F}^{\eps''}_{even}(\xi,\tilde{\eta})= 
\hat{F}^{\eps''}_{even}(\xi,{\eta}) + \int_\eta^{\tilde{\eta}} 
(&\Theta^{\eps"}(\xi,s) \delta_\xi \hat{F}_{even}^{\eps"}(\xi,s) + 
2i\cdot\Xi^{\eps"}(\xi,s) I_\xi \hat{F}_{even}^{\eps"}(\xi,s) \\
&+ (\eps'')^2 (R(\hat{F}_{even}^{\eps''}; \xi+\epss'',s,\eps'')- 
R(\hat{F}_{even}^{\eps"}; \xi-\epss'',s,\eps''))) ds +\Or(\eps'').
\end{align*}
In the limit $\eps''\to 0$ we therefore obtain
\begin{equation*}
 f_e(\xi,\tilde{\eta})= f_e(\xi,{\eta}) + \int_\eta^{\tilde{\eta}} 
(\Theta(\xi,s) \partial_\xi f_e(\xi,s) + 2i\cdot\Xi(\xi,s) f_e(\xi,s)) ds 
\end{equation*}
Of course, $\hat{F}^{\eps''}_{odd}$ and $f_o$ satisfy analogous equations. 
Consequently, $f_e$ and $f_o$ are differentiable in $\eta$ and satisfy
\begin{equation}\label{eq:feo}
 \partial_\eta f_I= \Theta\cdot \partial_\xi f_I +2i\cdot\Xi\cdot f_I,
\end{equation}
where $I\in\{e,o\}$.

Now estimate~\eqref{eq:estFn2} together with property~(6) of 
Lemma~\ref{lemPropNorm} imply that for fixed $\eta\in[-r,r]$ the functions 
$f_{e/o}(\eta,\xi)$ extend $\xi$-analytically to $\{\xi\in\C : 
\text{dist}(\xi,[-(r-|\eta|),r-|\eta|])\leq 4r-2|\eta| \}$ where they are 
bounded by $A$. Note that the analytic continuations also solve~\eqref{eq:feo}. 
So $f_e$ and $f_o$ are smooth with respect to $\eta$, as any $\eta$-derivative 
can be expressed in terms of $\xi$-derivatives and compositions with analytic 
functions. Finally, we deduce that $f_e$ and $f_o$ are the unique solution 
of~\eqref{eq:feo} with initial condition $f_e(\xi,0)=f_o(\xi,0)=f(\xi,0)$ and 
therefore satisfies $f_e=f_o=f$ on $\Diamond_r$.
\end{proof}

\subsection{Approximation in $C^0$}\label{secApprox}
Let $f^\eps$ be the restriction of the smooth solution $f$ 
of~\eqref{eq:evolvef} to $\Diamond^\eps_r$ and let $W^\eps=F^\eps-f^\eps$ 
denote the deviation of $F^\eps$ from $f$. 

\begin{lemma}\label{lemAbsBound}
 The difference $W^\eps=F^\eps-f^\eps$ is abolutely bounded on $\Diamond^\eps_r$  by $C \eps^2$, where the constant $C$ only depends on $f$, $r$ and $\Xi$.
\end{lemma}

\begin{proof}
Let $\rho_n=\mu_n-r=3r-|n|\eps$. Using estimate~(2) of Lemma~\ref{lemPropNorm}, it suffices to bound $\|W^\eps_{n}\|_{\rho_n}$.
In the following, we will calculate $\eps$-independent bounds on
\begin{equation}\label{eq:Ln}
 {\cal L}_n :=\|W^\eps_{n+1}\|_{\rho_n} + \|W^\eps_{n}\|_{\rho_n},
\end{equation}
which is defined for $\frac{\eps}{3}|n|\leq r$. Again, only $n\geq 0$ is considered here 
and the case $n\leq 0$ is left to the reader. In particular, it will be shown 
that
\begin{equation}\label{eq:Ln+1}
 {\cal L}_{n+1} \leq (1+\eps C_\Xi){\cal L}_{n} +(P^*+Q^*)\eps^3 \qquad 
\text{and}\qquad {\cal L}_0\leq T^*\eps^2.
\end{equation}
By the standard Gronwall estimate this leads to
\begin{equation}
 {\cal L}_n\leq (P^*+Q^*+T^*)\text{e}^{rC_\Xi} \; \eps^2.
\end{equation}

First, observe that $W_0^\eps\equiv 0$ and $\|W_1^\eps\|_{3r}\leq T^*\eps^2$ 
for some constant $T^*$ by our initial conditions~\eqref{eq:init1} 
and~\eqref{eq:init2}. In particular, due to~\eqref{eq:evolvef} and 
the analyticity of $f$ we have
\begin{equation*}
|f(\xi,0) + \frac{\eps}{2}\partial_\eta f(\xi,0)
-f(\xi,\epss)| \leq \frac{\eps^2}{8}\sup_{0\leq s\leq \epss}|\partial_\eta^2 
f(\xi,s)|.
\end{equation*}
Similar estimates $\|W_0^\eps\|_{3r}\leq \hat{T}^*\eps^2$ and $\|W_1^\eps\|_{3r}\leq \hat{T}^*\eps^2$ hold for the alternative initial conditions in Remark~\ref{remAltIni}.

Let $n\geq 1$. We use the same ideas for the estimates as for the proofs of Lemma~\ref{lemConsist} and Lemma~\ref{lemFeps}. 
\begin{align}
 {\cal L}_{n+1} &\leq \|W_{n+1}^\eps\|_{\rho_{n+1}} 
+\|W_n^\eps\|_{\rho_{n+1}} +\eps \|\Theta_{n+1}^\eps\|_{\rho_{n}} 
\|\delta_\xi W_{n+1}^\eps \|_{\rho_{n+1}} + 
\eps C_\Xi\|W_{n+1}^\eps\|_{\rho_{n+1}} \label{eq:LA} \\
&\quad +\eps^3\|R_{n+1}(F_{n+1}^\eps; 
\cdot+\epss)- R_{n+1}(F_{n+1}^\eps; \cdot-\epss) \|_{\rho_{n+1}} \label{eq:LB}\\
&\quad + \eps \|(\delta_\eta f^\eps)_{n+1} - M^\eps_{n+1} \delta_\xi 
f_{n+1}^\eps - 2i\cdot\Xi^\eps_{n+1} \cdot I_\xi f^\eps_{n+1} \|_{\rho_{n+1}}. 
\label{eq:LC}
\end{align}
We estimate the expressions in~\eqref{eq:LA}, \eqref{eq:LB}, and~\eqref{eq:LC} 
separately. 

First observe that by properties~(4) and~(1) of 
Lemma~\ref{lemPropNorm} and using $\rho_{n+1}+\eps=\rho_n$ we have
\begin{align*}
 &\|W_{n+1}^\eps\|_{\rho_{n+1}} 
+\|W_n^\eps\|_{\rho_{n+1}} +\eps \|\Theta_{n+1}^\eps\|_{\rho_{n}} 
\|\delta_\xi W_{n+1}^\eps \|_{\rho_{n+1}} + 
\eps C_\Xi\|W_{n+1}^\eps\|_{\rho_{n+1}} \\
&\leq \|W_n^\eps\|_{\rho_{n}}
+ \|W_{n+1}^\eps\|_{\rho_{n+1}+\eps} +\eps C_\Xi\|W_{n+1}^\eps\|_{\rho_{n}}\\
&\leq (1+\eps C_\Xi) {\cal L}_n.
\end{align*}

As in Section~\ref{secExist}, we deduce from~\eqref{eq:estf} and~\eqref{eq:estFn} that   
\begin{align*}
\|R_{n+1}(F_{n+1}^\eps; 
\cdot+\epss)- R_{n+1}(F_{n+1}^\eps; \cdot-\epss) \|_{\rho_{n+1}}
\leq 2A^3\Gamma {\cal R}=: P^*.
\end{align*}

Finally, the estimates in the proof of Lemma~\ref{lemConsist} and 
property~(5) of Lemma~\ref{lemPropNorm} show that~\eqref{eq:LC} is bounded,
\begin{equation*}
  \|(\delta_\eta f^\eps)_{n+1} - M^\eps_{n+1} \delta_\xi 
f_{n+1}^\eps - 2i\cdot\Xi^\eps_{n+1} \cdot I_\xi f^\eps_{n+1} \|_{\rho_{n+1}}
\leq Q^* \eps^2,
\end{equation*}
where the constant $Q^*$ only depends on $f$ and some $\rho'>2r$.

This shows that~\eqref{eq:Ln+1} holds and finishes the proof.
\end{proof}

\subsection{Smooth convergence}
Given a function $W$ on $\Diamond_r^\eps$, higher difference quotients 
$\delta_\eta^k\delta_\xi^m W$ are defined on the sublattice
\begin{equation*}
 \Diamond_r^{\eps, k+m}=\begin{cases} \Diamond^\eps_{r-(k+m)\epss} &\text{if } 
k+m \text{ is even},\\[2ex]
\{(\xi,\eta)\in\Diamond_{r-(k+m+1)\epss} : \epss+\xi+\eta\in\eps \Z\} &\text{if 
} k+m \text{ is odd} \end{cases}
\end{equation*}

The goal of this section is to prove that $F^\eps$ converges to $f$ in $C^\infty$, i.e.\ with all discrete partial derivatives. Together with the Arzela-Ascoli theorem and the result of the previous section, this is a consequence of the following lemma.
\begin{lemma}
Let $f$ be the smooth solution of~\eqref{eq:evolvef} and let $F^\eps$ be the corresponding discrete solution from Lemma~\ref{lemFeps}.
For $k,m\geq 0$ there are constants  $C_{k,m}>0$ such that
\begin{equation}\label{eq:estCinfty}
 \sup_{\Diamond_r^{\eps, k+m}}|\partial_\eta^k\partial_\xi^m f 
-\delta_\eta^k\delta_\xi^m F^\eps|\leq C_{k,m} \eps^2.
\end{equation}
\end{lemma}

Since $f,u,v$ are smooth on $\Diamond_r$, we may interchange partial 
derivatives and difference quotients with an error of order $\Or(\eps^2)$, see~\cite[Lemma~5.5]{Ma05} for a proof:
\begin{equation*}
 \sup_{\Diamond_r^{\eps, k+m}}|\partial_\eta^k\partial_\xi^m f 
-\delta_\eta^k\delta_\xi^m f^\eps|\leq C_{k,m}^* \eps^2.
\end{equation*}
Thus, for proving~\eqref{eq:estCinfty} it is sufficient to show that for all 
$k,m\geq 0$ we have
\begin{equation}\label{eq:estsmooth}
  \sup_{\Diamond_r^{\eps, k+m}}| \delta_\eta^k\delta_\xi^m W^\eps|\leq 
C_{k,m}' \eps^2,
\end{equation}
where $W^\eps=F^\eps-f^\eps$ 
denotes the deviation of $F^\eps$ from $f$ as in the previous section. 

To this end, we introduce further submultiplicative norms for functions 
$W:\Diamond_r^\eps\to\R$ by
\begin{equation*}
 \|W\|_\nu^{(N)} =\sum\limits_{n=0}^N 
\sum\limits_{m=1}^{\lfloor 2r/\eps \rfloor-n} 
\frac{(\widetilde{B} \nu)^{m+n}}{m!n!} \sup_{(\xi,\eta)\in \Diamond_r^{\eps, 
m+n}} 
|\delta_\eta^n\delta_\xi^m W(\xi,\eta)|.
\end{equation*}
Note that Lemma~\ref{lemNorm2} still applies and also~\eqref{eq:estf} and the 
discrete Cauchy estimate (property~(4) of Lemma~\ref{lemPropNorm}) still hold. 
The restriction estimate (property~(5) of Lemma~\ref{lemPropNorm}) reads as 
follows:
If $\tilde{w}$ is a smooth function such that $\tilde{w}(\cdot,\eta)$ is 
analytic on $B_\rho([-\rho',\rho'])= \{ \xi\in\C : 
\text{dist}(\xi,[-\rho',\rho']) \leq\rho\}$ for all $\eta\in[-r,r]$, the for 
$0<\nu<\rho$ there holds
\begin{equation}\label{eq:restrict2}
 \|\tilde{w}^\eps\|_\nu^{(N)} \leq {\cal C} \max_{n\leq N} \sup_{|\eta|\leq r} 
\sup_{\xi\in B_\rho([-\rho',\rho'])} |\partial_\eta^n \tilde{w}(\xi,\eta)|,
\end{equation}
where $\tilde{w}^\eps$ denotes the restriction of $\tilde{w}$ to 
$\Diamond_r^\eps$ and $\cal C$ depends on the ratio $\nu/\rho$.

We take as constant $\widetilde{B}=C_M^2$, where 
\begin{equation*}
 C_M= \sup_{\eps>0} \sum\limits_{n=0}^\infty 
\sum\limits_{m=1}^{\lfloor 2r/\eps \rfloor-n} \
\frac{r^{m+n}}{m!n!} \sup_{(\xi,\eta)\in \Diamond_r^{\eps, 
m+n}} |\delta_\eta^n\delta_\xi^m M^\eps (\xi,\eta)|.
\end{equation*}
This constant is finite as $M^\eps$ is analytic. For further use, we 
define the constant
\begin{equation*}
 \hat{C}_\Xi= \sup_{\eps>0} \sum\limits_{n=0}^\infty
\sum\limits_{m=1}^{\lfloor 2r/\eps \rfloor-n} \
\frac{(\hat{r}C_M)^{m+n}}{m!n!} \sup_{(\xi,\eta)\in \Diamond_r^{\eps, 
m+n}} |\delta_\eta^n\delta_\xi^m \Xi^\eps (\xi,\eta)|,
\end{equation*}
which is well-defined for some suitable $0<\hat{r}\leq r$ as $\Xi^\eps$ is also 
analytic. 

We will show inductively that 
\begin{equation}\label{eq:estwN}
 \|{W}^\eps\|_{\nu_N}^{(N)} \leq C_N \eps^2
\end{equation}
for suitable constants $C_N$, where $\nu_N=\frac{\nu_0}{C_M(N+1)}$ and 
$0<\nu_0<\hat{r}/C_M$ is chosen later. This estimate 
implies~\eqref{eq:estsmooth} with 
$C_{k,m}'=C_N m!k!(\widetilde{B}\nu_N)^{-k-m}$ and suitable $N$.

For $N=0$, inequality~\eqref{eq:estwN} follows from Section~\ref{secApprox} as 
$\nu_0<r\leq \rho_n$. Now assume that~\eqref{eq:estwN} holds for some $N\geq 
0$. Then we estimate similarly as in the previous section, using the properties of the submultiplicative norms and the discrete Cauchy estimate:
\begin{align*}
 \|{W}^\eps\|_{\nu_{N+1}}^{(N+1)} &\leq \|W^\eps\|_{\nu_{N+1}}^{(N)} +
\frac{C_M \nu_{N+1}}{N+1} \|\delta_\eta W^\eps\|_{\nu_{N+1}}^{(N)} \\
&\leq  \underbrace{\|W^\eps\|_{\nu_{N+1}}^{(N)} + \frac{\nu_0}{(N+1)(N+2)} 
\|M^\eps \delta_\xi W^\eps\|_{\nu_{N+1}}^{(N)}}_{\leq   
\|W^\eps\|_{\nu_{N}}^{(N)}} 
 + \underbrace{\frac{\nu_0}{(N+1)(N+2)} \|2i\;\Xi^\eps I_\xi 
W^\eps\|_{\nu_{N+1}}^{(N)}}_{\leq \frac{2\nu_0\hat{C}_\Xi}{(N+1)(N+2)}
\|W^\eps\|_{\nu_{N}}^{(N)}}\\
&\quad +\frac{\nu_0}{(N+1)(N+2)} \|M^\eps \delta_\xi f^\eps +2i\;\Xi^\eps 
I_\xi f^\eps -\delta_\eta f^\eps \|_{\nu_{N+1}}^{(N)} \\
&\quad +\frac{\nu_0}{(N+1)(N+2)} \eps^2 \|R(F^\eps; 
\cdot+\epss,\eta,\eps)- R(F^\eps; \cdot-\epss,\eta,\eps)\|_{\nu_{N+1}}^{(N)}.
\end{align*}
The functions
\begin{align*}
 A_1^\eps &= M^\eps \delta_\xi f^\eps - \Theta \partial_\xi f, &\qquad
 A_2^\eps &= 2\Xi^\eps \delta_\xi f^\eps - 2\Xi \partial_\xi f, &\qquad
 A_3^\eps &= \delta_\eta f^\eps - \partial_\eta f \\
\end{align*}
are all analytic for every $\eps>0$ and of order $\Or(\eps^2)$, see the proof of Lemma~\ref{lemConsist}. The restriction estimate~\eqref{eq:restrict2}
implies that 
\begin{equation*}
 \|M^\eps \delta_\xi f^\eps +2i\Xi^\eps 
I_\xi f^\eps -\delta_\eta f^\eps \|_{\nu_{N}}^{(N)} \leq \eps^2 {\cal C} 
\max_{n\leq N} \sup_{|\eta|+|\xi|\leq r} (|\partial_\eta^n A_1^\eps|+ 
|\partial_\eta^n A_2^\eps|+ |\partial_\eta^n A_3^\eps|) \leq \widehat{C}_N \eps^2.
\end{equation*}
Furthermore, the restriction estimate~\eqref{eq:restrict2} and Lemma~\ref{lemNorm2} imply that
\begin{equation*}
\|R(F^\eps; \cdot+\epss,\eta,\eps)- R(F^\eps; 
\cdot-\epss,\eta,\eps)\|_{\nu_{N+1}}^{(N)} \leq \Gamma A^3 {\cal C} C_R
\end{equation*}
for some constant $C_R$. In order to justify the application of 
Lemma~\ref{lemNorm2}, we need to guarantee that $\|F^\eps\|_{\nu_N}^{(N)}\leq 
A$ for suitable (possibly diminished) $r$.

Choose  $\nu_0<\hat{r}/C_M$ satisfying additionally $\nu_0 (2\hat{C}_\Xi +1) 
\sum_{n=0}^\infty 1/(n+1)^2 \leq 1$. Furthermore, we choose $\eps>0$ small 
enough such that $\eps^2 \Gamma A^3  {\cal C} C_R <A/4$. We will show by induction that  
\begin{equation*}
\|F^\eps\|_{\nu_{N}}^{(N)} \leq \frac{A}{4} \text{exp}\left(\sum_{n=0}^{N-1} 
\frac{\nu_0 (2\hat{C}_\Xi +1)}{(n+1)^2}\right)
\end{equation*}
holds for all $N\geq 0$. By our results in Section~\ref{secExist} $\|F^\eps\|_{\nu_{0}}^{(0)}\leq A/4$ holds. Estimating along the same lines as above, we obtain that
\begin{align*}
 \|F^\eps\|_{\nu_{N+1}}^{(N+1)} 
&\leq  \underbrace{\|F^\eps\|_{\nu_{N+1}}^{(N)} + \frac{\nu_0}{(N+1)(N+2)} 
\|M^\eps \delta_\xi F^\eps\|_{\nu_{N+1}}^{(N)} +\frac{\nu_0}{(N+1)(N+2)} 
\|2i\Xi^\eps I_\xi F^\eps\|_{\nu_{N+1}}^{(N)}}_{\leq   
\left(1+ \frac{2\nu_0\hat{C}_\Xi}{(N+1)(N+2)}\right) \|F^\eps\|_{\nu_{N}}^{(N)}} \\
&\quad +\frac{\nu_0}{(N+1)(N+2)} \underbrace{\eps^2 \|R(F^\eps; 
\cdot+\epss,\eta,\eps)- R(F^\eps; \cdot-\epss,\eta,\eps)\|_{\nu_{N+1}}^{(N)}}_{ 
\leq  A/4}\\
&\leq  \frac{A}{4} \text{exp}\left(\sum_{n=0}^{N-1} 
\frac{\nu_0 (2\hat{C}_\Xi +1)}{(n+1)^2}\right)
(1+\frac{2\nu_0\hat{C}_\Xi}{(N+1)(N+2)} + \frac{\nu_0}{(N+1)(N+2)} ) \\
&\leq \frac{A}{4} \text{exp}\left(\sum_{n=0}^{N-1} 
\frac{\nu_0 (2\hat{C}_\Xi +1)}{(n+1)^2}\right)
(1+\frac{\nu_0(2\hat{C}_\Xi+1)}{(N+1)(N+2)}) \leq \frac{A}{4} 
\text{exp}\left(\sum_{n=0}^{N} \frac{\nu_0 (2\hat{C}_\Xi +1)}{(n+1)^2} \right),
\end{align*}
where we have applied the induction hypothesis.

\section{Convergence of the CR-mappings $G$ and  of the discrete minimal surfaces from Bj\"orling data}\label{secConvMinimal}

Given the approximation of $f$ by $F^\eps$, we deduce the convergence of the related geometric notions $\alpha$ and $\beta$ which finally shows the convergence of the CR-mapping $G_{m,n}$ for suitable initial conditions. This directly shows the approximation of the smooth solution to the Björling problem by the discrete minimal surfaces obtained from $G_{m,n}$.

\subsection{Convergence of the discrete derivatives $\alpha$ and 
$\beta$}\label{secConvalphabeta}
The proofs in Sections~\ref{secCR} and~\ref{secConvF} show that given an analytic function $f$ on 
a parametrized curve $\gamma:(-a,a)\to\Omega$ as detailed in 
Section~\ref{secDeriv}, we can locally approximate the values of $f$ (and its derivatives) using a function $F^\eps$ on a suitable lattice. 
In the following, we deduce the convergence for the auxiliary functions $\alpha$ and $\beta$, introduced in Section~\ref{secDeriv}, to the given function $g'$.

\subsubsection{Convergence of $\log\alpha$ to $\log g'$}
Recall that by~\eqref{eq:defQ} and with our ansatz $Q=\text{e}^{\eps F}$ we can write
\begin{align*}
 F_{m-\halb,n+\halb}=\frac{1}{\eps} \log Q_{m-\halb,n+\halb}
&= -\frac{1}{\eps} \left[ \log\frac{G_{m,n+1}-G_{m,n}}{p_{m,n+1}-p_{m,n}}
-\log\frac{G_{m,n}-G_{m-1,n}}{p_{m,n}-p_{m-1,n}} \right] \\
&=-\frac{1}{\eps} \left[\log \alpha_{m,n+\halb} -\log \beta_{m-\halb,n}\right],
\end{align*}
where we now associate the values of $\alpha$ and $\beta$ to the edge-midpoint $(m,n+\scriptstyle\halb \textstyle )$ and $(m-\scriptstyle\halb \textstyle ,n)$ resp. Using the identity
\begin{align*}
\text{e}^{\eps F_{m-\halb,n+\halb}}= Q_{m-\halb,n+\halb} &= 
\frac{\beta_{m-\halb,n}}{\alpha_{m,n+\halb}} 
\overset{\eqref{eq:Q2}}{=} 
\frac{\alpha_{m-1,n+\halb}}{\beta_{m-\halb,n+1}} 
\end{align*}
 we deduce that
\begin{align}
 \delta_\xi^\eps \log\alpha(m-{\textstyle\halb},n) &= \frac{1}{\eps} \left( 
\log\alpha_{m,n+\halb} -\log\beta_{m-\halb,n} 
+\log\beta_{m-\halb,n} -\log\alpha_{m-1,n-\halb}\right)\notag \\ 
&= -F_{m-\halb,n+\halb} -F_{m-\halb,n-\halb}.\label{eq:Falpha}
\end{align}
Here we use discrete partial derivatives 
\begin{equation*}
 \delta_\xi^\eps H (m-{\textstyle\halb},n) = \frac{1}{\eps}(H_{m,n+\halb} -H_{m-1,n-\halb})
\end{equation*}
of a discrete function $H$, which is --- like $\alpha$ --- defined on the midpoints of the 'vertical' edges of the grid $\Omega^\eps$.
Extending the relation between $\alpha$ and $F$, we can define $\delta_\xi^\eps \log\alpha$ on the shifted double cone 
$\ddiamond^{{\eps}}_r=\{(\xi,\eta)\,|\, (\xi- \epsss,\eta- \epsss) \in \Diamond^\eps_r \wedge (\xi+ \epsss,\eta+ \epsss) \in \Diamond^\eps_r\}$ 
by
\begin{equation}\label{eq:alphaFeps}
 \delta_\xi^\eps \log\alpha(\xi,\eta)= -F^\eps(\xi+\epsss,\eta+\epsss) -F^\eps(\xi-\epsss,\eta-\epsss),
\end{equation}
Starting from a given initial value, we also may consider $\alpha$ as a function on $\ddiamond^{{\eps}}_r$. Using this continuation $\alpha^\eps$, we can now show its smooth convergence.
Set $h(\xi,\eta):= \log g'\circ p(\xi,\eta)= \log g'\big(u(\xi-\eta)+iv(\xi+\eta)\big)$, so $f= 
-\frac{1}{2} \partial_\xi h$, and denote its restriction to $\ddiamond^{{\eps}}_r$ 
by $h^\eps$. Then 
\begin{align*}
 \delta_\xi (\log\alpha^\eps-  h^\eps) (\xi,\eta) 
=&
 -F^\eps(\xi+\epsss,\eta+\epsss) -F^\eps(\xi-\epsss,\eta-\epsss) + 
f^\eps(\xi+\epsss,\eta+\epsss) + f^\eps(\xi-\epsss,\eta-\epsss) \\
 & - \delta_\xi^\eps h^\eps (\xi,\eta) + [\partial_\xi h]^\eps (\xi,\eta) \\
& +\frac{1}{2} ([\partial_\xi h]^\eps(\xi+\epsss,\eta+\epsss) + 
[\partial_\xi h]^\eps(\xi-\epsss,\eta-\epsss)) -[\partial_\xi h]^\eps (\xi,\eta)
\end{align*}
As the functions $- \delta_\xi h^\eps  + [\partial_\xi h]^\eps$ and $\frac{1}{2} 
([\partial_\xi h]^\eps(\xi+\epsss,\eta+\epsss) + 
[\partial_\xi h]^\eps(\xi-\epsss,\eta-\epsss)) -[\partial_\xi 
h]^\eps(\xi,\eta)$ are analytic in $\xi$ and $\eta$, we deduce similarly as 
above from the restriction estimate that
\begin{align*}
 &\|\delta_\xi h^\eps - [\partial_\xi h]^\eps \|^{(N)}_1\leq C_{1,N}\eps^2, \\
& \|\frac{1}{2} ([\partial_\xi h]^\eps(\cdot+\epsss,\cdot+\epsss) + 
[\partial_\xi h]^\eps(\cdot-\epsss,\cdot-\epsss)) -[\partial_\xi h]^\eps 
\|^{(N)}_1\leq C_{2,N}\eps^2.
\end{align*}
Together with~\eqref{eq:estwN}, this implies that $\|\delta_\xi 
(\log\alpha^\eps- h^\eps)\|^{(N)}_1\leq C_{3,N}\eps^2$ for some constants $C_{3,N}$.

From~\eqref{eqalpha-} and~\eqref{eqalpha+} we deduce with an analogous definition of $\delta_\eta^\eps$ that 
\begin{align*}
 \delta_\eta^\eps \log\alpha(m-{\textstyle\halb},n) &= -(M_{m-1,n}+2i\eps\Xi_{m-1,n})F_{m-\halb,n+\halb} - (M_{m-1,n-1}+2i\eps\Xi_{m-1,n-1}) F_{m-\halb,n-\halb} \\
 &\quad +\eps^2(R_{m-\halb,n+\halb}(F;\eps) + R_{m-\halb,n-\halb}(F;\eps) ).
\end{align*}
Using the continuation $\alpha^\eps$ and $h^\eps$ as above, we obtain
\begin{align*}
& \delta_\eta (\log\alpha^\eps- h^\eps) (\xi,\eta) \notag \\
&\quad = -\frac{1}{2}\left( \Theta^\eps(\xi+\epsss,\eta+\epsss) + 
\Theta^\eps(\xi-\epsss,\eta-\epsss) \right) (F^\eps(\xi+\epsss,\eta+\epsss) 
+F^\eps(\xi-\epsss,\eta-\epsss)) -\delta_\eta h^\eps(\xi,\eta) \label{eq:detah} 
\\
&\qquad -\frac{1}{2}\left( \Theta^\eps(\xi+\epsss,\eta+\epsss) - 
\Theta^\eps(\xi-\epsss,\eta-\epsss) \right) (F^\eps(\xi+\epsss,\eta+\epsss) 
-F^\eps(\xi-\epsss,\eta-\epsss)) \notag\\
&\qquad +\eps^2 (R(F^\eps;\xi+\epsss,\eta+\epsss,\eps) +
R(F^\eps;\xi-\epsss,\eta-\epsss,\eps)),\notag
\end{align*}
where we have used the definitions of $\Theta^\eps$ by~\eqref{eq:defThetaeps} 
and of $R$ by~\eqref{eq:defR}. We estimate the terms in the above three lines separately.
The results of Section~\ref{secApprox} imply that 
\begin{align*}
\|R(F^\eps;\xi+\epsss,\eta+\epsss,\eps) +
R(F^\eps;\xi-\epsss,\eta-\epsss,\eps)\|^{(N)}_1&\leq C_{4,N} \\
\text{and }\quad \|F^\eps(\xi+\epsss,\eta+\epsss) -F^\eps(\xi-\epsss,\eta-\epsss)|^{(N)}_1 
& \leq \|F^\eps(\xi+\epsss,\eta+\epsss) - 
f^\eps(\xi+\epsss,\eta+\epsss)\|^{(N)}_1 \\
&\qquad + \|F^\eps(\xi-\epsss,\eta-\epsss) 
-f^\eps(\xi-\epsss,\eta-\epsss)\|^{(N)}_1 \\
&\qquad + \|f^\eps(\xi+\epsss,\eta+\epsss) 
-f^\eps(\xi-\epsss,\eta-\epsss)\|^{(N)}_1 \\
&\leq C_{5,N}\eps
\end{align*}
for some constants $C_{4,N}$ and $C_{5,N}$.
As $\Theta^\eps$ is analytic in $\xi$ and $\eta$, we deduce that 
\begin{equation*}
\|\Theta^\eps(\xi+\epsss,\eta+\epsss) -\Theta^\eps(\xi-\epsss,\eta-\epsss) 
\|^{(N)}_1\leq C_{6,N}\eps.
\end{equation*}
In order to estimate the term in the first line above, observe that $h$ satisfies the evolution equation
\begin{equation}\label{eq:evolveh}
\partial_\eta h = \Theta\cdot\partial_\xi h.
\end{equation}
Thus, $\partial_\eta h = -2\Theta\cdot f$. 
Using the analyticity of $\Theta^\eps$, $\Theta$, $f$ and $h$ and 
the estimates derived in Section~\ref{secApprox} we finally deduce by similar estimates that
\begin{align*}
 &\|\frac{1}{2}\left( \Theta^\eps(\xi+\epsss,\eta+\epsss) + 
\Theta^\eps(\xi-\epsss,\eta-\epsss) \right) (F^\eps(\xi+\epsss,\eta+\epsss) 
+F^\eps(\xi-\epsss,\eta-\epsss)) +\delta_\eta h^\eps(\xi,\eta)\|^{(N)}_1\\
&\leq C_{7,N}\eps^2.
\end{align*}
Combining all previous 
estimates shows that $\|\delta_\eta (\log\alpha^\eps- h^\eps)\|^{(N)}_1 
\leq C_{8,N}\eps^2$ for some constants $C_{8,N}$.

For an estimate on $\log\alpha^\eps- h^\eps$, we need another ingredient. 
\begin{lemma}[{\cite[Lemma~5.4]{Ma05}}]
 Let $W:\Diamond_r^\eps \to\C$ and $N\geq 0$. Then
\begin{equation*}
 \|W\|^{(N+1)}_1 \leq C_r(|W(z_0)|+ \|\partial_\xi W\|^{(N)}_1+ \|\partial_\eta 
W\|^{(N)}_1),
\end{equation*}
where $z_0\in\Diamond_r^\eps$ is an arbitrary point.
\end{lemma}

By this lemma we only need one initial value
$\alpha^\eps(\xi_0,\epsss)$ (which we will obtain from two initial values for 
$G_{m,n}$ in the next section), which is sufficiently close to $h(\xi_0,\epsss)$, 
that is $|\alpha^\eps(\xi_0,\epsss)-h(\xi_0,\epsss)| \leq C_{9}\eps^2$. 
This then implies
\begin{equation}\label{eq:estlogalpha}
 \|\log\alpha^\eps- h^\eps\|^{(N)}_1 \leq C_{10,N}\eps^2.
\end{equation}

\subsubsection{Convergence of $\log\beta$ to $\log g'$}
Similar reasoning as in the case of $\alpha$ applies to the consideration of $\log\beta$ and we can derive similar estimates as for the discrete derivatives of 
$\log\alpha^\eps-h^\eps$ for those of $\log\beta^\eps 
-h^\eps$. Alternatively, notice that with the above notations
\begin{align*}
 \log\beta^\eps(\xi,\eta)- h^\eps(\xi,\eta) &\stackrel{\eqref{eq:defQ}}{=} 
 \log\alpha^\eps(\xi+\epss,\eta)+\eps F^\eps(\xi+\epsss,\eta+\epsss)
 -h^\eps(\xi,\eta) \\
&= \log\alpha^\eps(\xi+\epss,\eta) -h^\eps(\xi+\epss,\eta) \\
&\quad + h^\eps(\xi+\epss,\eta) -h^\eps(\xi,\eta) -\epss [\partial_\xi 
h]^\eps(\xi+\epsss,\eta) \\
&\quad +\epss ([\partial_\xi 
h]^\eps(\xi+\epsss,\eta) -[\partial_\xi h]^\eps(\xi+\epsss,\eta+\epsss)) \\
&\quad +\eps (-f^\eps (\xi+\epsss,\eta+\epsss)+ 
F^\eps(\xi+\epsss,\eta+\epsss)). 
\end{align*}
Again we estimate the terms in the different lines separately. By the results of Section~\ref{secApprox} we have $\|f^\eps 
(\xi+\epsss,\eta+\epsss)- F^\eps(\xi+\epsss,\eta+\epsss)\|^{(N)}_1  \leq 
C_{9,N}\eps^2$. As above, the analyticity of $h$ implies that 
$\|h^\eps(\xi+\epss,\eta) -h^\eps(\xi,\eta) -\epss [\partial_\xi 
h]^\eps(\xi+\epsss,\eta)\|^{(N)}_1  \leq C_{10,N}\eps^2$ and $\|[\partial_\xi 
h]^\eps(\xi+\epsss,\eta) -[\partial_\xi 
h]^\eps(\xi+\epsss,\eta+\epsss)\|^{(N)}_1  \leq C_{11,N}\eps$. Therefore, we 
can deduce from~\eqref{eq:estlogalpha} that
\begin{equation}\label{eq:estlogbeta}
 \|\log\beta^\eps- h^\eps\|^{(N)}_1 \leq C_{11,N}\eps^2.
\end{equation}

\subsubsection{Convergence of $\alpha$ and $\beta$ to $g'$}
As $h=\log g'$, an application of~\eqref{eq:estdiff} to~\eqref{eq:estlogalpha} 
and~\eqref{eq:estlogbeta} shows that 
\begin{align}\label{eq:estalphabeta}
 \|\alpha^\eps- g' \|^{(N)}_1  &\leq C_{12,N}\eps^2,  &
\|\beta^\eps -g' \|^{(N)}_1  &\leq C_{12,N}\eps^2 
\end{align}
for a suitable initial value $\alpha^\eps(\xi_0,\epsss)$. This means in particular, that the difference quotients
\begin{equation*}
 \frac{G_{m,n+1}-G_{m,n}}{p_{m,n+1}-p_{m,n}}\qquad
\text{ and }\qquad
\frac{G_{m,n}-G_{m-1,n}}{p_{m,n}-p_{m-1,n}}, 
\end{equation*} 
defined on the midpoints of the edges of $\Omega^\eps$, approximate in $C^\infty$ the values of $g'$.

\subsection{Cauchy data for $G_{m,n}$ obtained from $\phi$, $G_0$ and one value of $\No_0$}\label{secConst2}

Having prepared our ingredients in the previous sections, we now sum up, how to
obtain our promised approximation results of Theorems~\ref{theo:MinimalConv} and~\ref{theo:ConvG}.

We start from Björling data as explained in the beginning of Section~\ref{secConstruction} and use mostly values of the auxialary function $f$ and its derivative. This consitutes an alternative approach instead of the construction presented in Section~\ref{secConstruction}.

In summary, our construction amounts to the following. See Figure~\ref{fig:initial} for a schematic sketch of the initial data detailed in the steps below.
\begin{enumerate}[(i)]
 \item 
Thanks to the relations detailed in Section~\ref{secPhi}, we can pass from the original functions $\F_0$ and $\No_0$ to the maps $g$ and $\phi=u+iv$, where $\phi$ gives rise to a coordinate transformation $p$, see~\eqref{eq:p1}. 
By assumption, these functions possess analytic extensions and allow to define the auxiliary function $f=-\frac{1}{2}\partial_\xi \log g'\circ p$, see~\eqref{eq:42}. 
\item We now choose a parameter $\eps$ and determine the rectangular lattice $p_{m,n}$ from $\phi$, see Section~\ref{secDisHolo}. 
\item
 We fix two initial values for $G_{m,n}$, namely
\begin{align}
 G_{0,0}& := g(\phi(0))= G_0(0) = \sigma\circ\No_0(0) \label{eq:initG1} \\[1ex]
G_{0,1}&
:= g(\phi(0))+i(v(\eps)-v(0)) g'(\phi(0)) + \frac{1}{2}i^2(v(\eps)-v(0))^2 g''(\phi(0)) \label{eq:initG2}  
\end{align}
Only here we use our assumption, that $0$ is contained in the domain of $\phi$. Of course, this may easily be adapted for arbitrary starting points $t_0$. 
Then we obtain from~\eqref{eq:defalphabeta} the value
\begin{equation}\label{eq:initalpha}
\alpha^\eps(\epsss,\epsss)= \frac{G_{0,1}- G_{0,0}}{i(v(\eps)- 
v(0))} = g'(p(\epsss,\epsss)) +\Or(\eps^2).
\end{equation}
So this is a suitable initial value such that~\eqref{eq:defalphabeta} holds because, due to 
the analyticity, $g'$ may be approximated by central differences in $C^\infty$ with an error of order $\eps^2$. 
\item
Note in particular that given the initial value $\alpha^\eps(\epsss,\epsss)$ as 
above, all values of $\alpha^\eps$ and $\beta^\eps$ on the initial 'zig-zag' are determined from $F^\eps$ and their relations to $Q^\eps$ as detailed in the previous section. The relevant values of $F^\eps$ in turn are just the initial values~\eqref{eq:init1} and~\eqref{eq:init2}. Rewriting these initial values in terms of  $\phi$ and $G_0=g\circ \phi$ as
\begin{align}
 F^\eps(t,0)&= f(t,0)=-\frac{1}{4} \frac{d}{dt} \log \frac{\dot{G}_0^2}{\dot{\phi}^2},\label{eq:init1l}\\
 F^\eps(t,\epss)&= -\frac{1}{4} \frac{d}{dt} \log \frac{\dot{G}_0^2}{\dot{\phi}^2} -
 \frac{\eps}{8}  \frac{d}{dt}\left(\frac{\overline{\dot\phi}}{\dot\phi} \cdot \frac{d}{dt} \log \frac{\dot{G}_0^2}{\dot{\phi}^2}\right), \label{eq:init2l}
\end{align}
we observe that they only depend on the given Björling data as detailed in Section~\ref{secPhi}. 

From this initial data for $F^\eps$ we obtain the values of $\alpha^\eps$ on the initial 'zig-zag' from~\eqref{eq:Falpha} and \eqref{eq:initalpha}. 
\item
By~\eqref{eq:defQ} and \eqref{eq:defalphabeta} we can finally deduce from~\eqref{eq:initG1}--\eqref{eq:initG2} initial values for $G^\eps$. 
\end{enumerate}

In this way, we have generated Cauchy data from which we determine our solution $G_{m,n}$ of the cross-ratio equation~\eqref{eq:devolve02}.
Additionally, we deduce from our construction that this solution $G_{m,n}$
approximates $g$ in $C^\infty$ with an error of order $\eps^2$. This finishes the proof of Theorem~\ref{theo:ConvG}.

\begin{figure}
\begin{tikzpicture}[rotate=-45,scale=0.75]
    \def\xmin{0}
    \def\xmax{5}
    \def\ymin{0}
    \def\ymax{5}
    \def\xnum{10}
    \def\ynum{10}
    \draw[help lines]
      \foreach \i in {1, ..., \ynum} {
        (\xmin, {\ymin + 0.095*(\i)*(\ymax - \ymin)})
        -- ++(\xmax - \xmin, 0)
      }
      \foreach \i in {1, ..., \xnum} {
({\xmin + 0.095*(\i)*(\xmax - \xmin)}, \ymin)
        -- ++(0, \ymax - \ymin)
      }
    ;
          \draw[->, >=stealth'] (-0.1,-0.1)  -- (5.4, 5.4) node(tline)[above right]
        {$t$};
        \draw[->, >=stealth'] (2.2,-2.4)  -- (-2.5, 2.3) node(etaline)[left]
        {$\eta$};
        
\draw [thick, blue] (0.095*5,0.095*5) -- (0.095*50,0.095*50);
\draw [thick, blue] (0.095*5,0.095*2*5) -- (0.095*45,0.095*50);
\draw [thick, red] (0.095*5*5,0.095*5*5) -- (0.095*5*5,0.095*30);
\draw [thick,dashed, red!50] (0.095*5*4,0.095*5*4) -- (0.095*5*4,0.095*5*5);
\draw [thick,dashed, red!50] (0.095*5*3,0.095*5*3) -- (0.095*5*3,0.095*5*4);
\draw [thick,dashed, red!50] (0.095*5*2,0.095*5*2) -- (0.095*5*2,0.095*5*3);
\draw [thick,dashed, red!50] (0.095*5*1,0.095*5*1) -- (0.095*5*1,0.095*5*2);
\draw [thick,dashed, red!50] (0.095*5*6,0.095*5*6) -- (0.095*5*6,0.095*5*7);
\draw [thick,dashed, red!50] (0.095*5*7,0.095*5*7) -- (0.095*5*7,0.095*5*8);
\draw [thick,dashed, red!50] (0.095*5*8,0.095*5*8) -- (0.095*5*8,0.095*5*9);
\draw [thick,dashed, red!50] (0.095*5*9,0.095*5*9) -- (0.095*5*9,0.095*5*10);

\filldraw[black] (0.095*5*5,0.095*5*5) circle (1.5pt);
\filldraw[black] (0.095*5*5,0.095*5*6) circle (1.5pt);
\draw[fill=white] (0.095*5*1,0.095*5*1) circle (1.5pt);
\draw[fill=white] (0.095*5*1,0.095*5*2) circle (1.5pt);
\draw[fill=white] (0.095*5*2,0.095*5*2) circle (1.5pt);
\draw[fill=white] (0.095*5*2,0.095*5*3) circle (1.5pt);
\draw[fill=white] (0.095*5*3,0.095*5*3) circle (1.5pt);
\draw[fill=white] (0.095*5*4,0.095*5*5) circle (1.5pt);
\draw[fill=white] (0.095*5*6,0.095*5*6) circle (1.5pt);
\draw[fill=white] (0.095*5*6,0.095*5*7) circle (1.5pt);
\draw[fill=white] (0.095*5*7,0.095*5*7) circle (1.5pt);
\draw[fill=white] (0.095*5*7,0.095*5*8) circle (1.5pt);
\draw[fill=white] (0.095*5*8,0.095*5*8) circle (1.5pt);
\draw[fill=white] (0.095*5*8,0.095*5*9) circle (1.5pt);
\draw[fill=white] (0.095*5*9,0.095*5*9) circle (1.5pt);
  \end{tikzpicture}
  \caption{Schematic sketch: Initial values for $F^\eps$  are taken from~\eqref{eq:init1l}--\eqref{eq:init2l} and indicated in blue. Initial values for $G$ according to~\eqref{eq:initG1}--\eqref{eq:initG2} are indicated as black dots and the resulting initial value $\alpha^\eps(\epsss,\epsss)$ by~\eqref{eq:initalpha} is marked in red. These initial values allow to obtain all remaining values for $\alpha(\cdot,\epsss)$ (dashed light red) and thus the necessary values for $G$ on the initial 'zig-zag' (white dots).}\label{fig:initial}
  \end{figure}
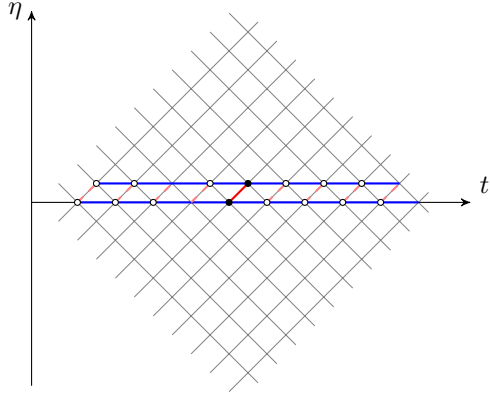

From this CR-mapping $G_{m,n}$ we construct a discrete minimal surface $\F_{m,n}$ for a suitably chosen starting point $\F_{0,0}$ from discrete integration of~\eqref{eq:discreteWeierstrass1}--\eqref{eq:discreteWeierstrass2}. As $G_{m,n}$ approximates the smooth function $g$ locally in $C^\infty$ with error of order $\eps^2$, we deduce
\begin{align*}
 \frac{\F_{m+1,n}-\F_{m,n}}{p_{m+1,n}-p_{m,n}}&= 
\Re\left[\frac{1}{g'}\rho(g)\right] +\Or(\eps^2)= \F_x+\Or(\eps^2), \\
\frac{\F_{m,n+1}-\F_{m,n}}{-i(p_{m,n+1}-p_{m,n})}&= 
-\Im\left[\frac{1}{g'}\rho(g)\right] +\Or(\eps^2)= \F_y+\Or(\eps^2) ,
\end{align*}
where we used the smooth Weierstrass representation~\eqref{eq:Weierstrass} and~\eqref{eq:defsigma}. 

This shows that the discrete minimal surface $\F_{m,n}$ locally approximates the smooth minimal surface $\F$ with an error of order~$\eps^2$ and thus proves Theorem~\ref{theo:MinimalConv}.

\begin{rmk}
 Alternatively to~\eqref{eq:initG2}, we could take as initial value
 \begin{align*}
G_{0,1}&
:= g(\phi(0))+i\eps \dot v(0) g'(\phi(0)) + \frac{1}{2}i^2\eps^2\dot v(0)^2 g''(\phi(0)) + \frac{1}{2}i\eps^2\ddot v(0) g'(\phi(0))
\end{align*}
\end{rmk}

\begin{rmk}
 For the initial values for $F^\eps$ in~\eqref{eq:init1l}--\eqref{eq:init2l} we have used derivatives of the given Björling data. Instead, we could use other choices as long as $\| 
F_0^\eps\|_{\mu_0}$ and $\|F_1^\eps\|_{\mu_0}$ are uniformly bounded in $\eps$ as well as $\|W_0^\eps\|_{3r}$ and $\|W_1^\eps\|_{3r}$ have bounds of order $\eps^2$ or $\eps$. In the latter case, we obtain an approximation error of order $\eps$.

A different choice for the initial values for $F^\eps$ may be based on knowing the exact values of the function $g$ in a (small) neighborhood of the curve $\gamma$:
\begin{align*}
F^\eps(t,0)&= \frac{1}{\eps} \left(\log\frac{g(u(t-\epss)+iv(t+\epss)) -g(u(t-\epss)+iv(t-\epss))}{i(v(t+\epss)-v(t-\epss))} \right. \\
&\qquad\quad \left.- \log\frac{g(u(t-\epss)+iv(t+\epss)) -g(u(t+\epss)+iv(t+\epss))}{u(t-\epss)-u(t+\epss)}\right), \\[0.5ex]
F^\eps(t,\epss)&= \frac{1}{\eps} \left(\log\frac{g(u(t)+iv(t+\eps)) -g(u(t)+iv(t))}{i(v(t+\eps)-v(t))} \right. \\
&\qquad\quad \left.- \log\frac{g(u(t-\eps)+iv(t)) -g(u(t)+iv(t))}{u(t)-u(t-\eps)}\right).
\end{align*}
In this case, it can easily be checked using a computer algebra program that the approximation error is of order $\eps$.
\end{rmk}

\begin{rmk}
 For the local approximation of the smooth function $g$ we use discrete holomorphic maps based on cross-ratio preservation (see Definition~\ref{defCRmap}), because the discrete minimal surfaces defined 
 in~\eqref{eq:discreteWeierstrass1}--\eqref{eq:discreteWeierstrass2} rely on CR-mappings. An analogous claim to Theorem~\ref{theo:ConvG} for discrete holomorphic maps based on the linear theory (see for example~\cite[Chap.~7]{BS08} and references therein), that is
 \begin{align*}
 \frac{f_{m,n+1}-f_{m-1,n}}{f_{m-1,n+1}-f_{m,n}} =\frac{p_{m,n+1}-p_{m-1,n}}{p_{m-1,n+1}-p_{m,n}},
 \end{align*}
 may be proved along the same lines as detailed in the previous sections.
\end{rmk}

\subsection{Cauchy data for $G_{m,n}$ obtained from $\No_0$, $\phi$ and $f$ and proofs of Theorems~\ref{theo:ConvG} and~\ref{theo:MinimalConv}}
Considering the construction procedure in Section~\ref{secConstruction}, observe that we only partially use the same initial values for $F^\eps$ as in Theorem~\ref{theoConvF}. The initial data for $\eta=0$ coincides with~\eqref{eq:init1}, but from our definitions for $G_{m,m}$ and $G_{m,m+1}$ and our ansatz in Section~\ref{secDeriv} we obtain further initial values for $F^\eps$ for $\eta=\epss$ which differ from~\eqref{eq:init2}. A Taylor approximation shows that these values are also suitable to apply Lemma~\ref{lemAbsBound} and thus deduce convergence analogously as in Section~\ref{secApprox} and~\ref{secConvF}. Therefore, we can use the values for $G_{m,n}$ defined in~\eqref{eq:defGm1}--\eqref{eq:defGm2} as our initial 'zig-zag'-curve. Evolution by~\eqref{eq:defdiscconf} produces a CR-mapping $G_{m,n}$ in a neighborhood of $G(t_0)$. 
This finishes the proof of Theorem~\ref{theo:ConvG}.

\section*{Acknowledgement}
This research was supported by the Deutsche Forschungsgemeinschaft (DFG --- German Research Foundation) --- Project-ID 195170736 --- TRR109 “Discretization in Geometry and Dynamics”.

\bibliographystyle{plain}
\bibliography{bjoerling}

\end{document}